\documentclass{siamltex}
\usepackage{amsmath}
\usepackage{amsfonts}
\usepackage{amssymb}

\DeclareFontFamily{U}{mathb}{\hyphenchar\font45}
\DeclareFontShape{U}{mathb}{m}{n}{ <5> <6> <7> <8> <9> <10> gen *
mathb <10.95> mathb10 <12> <14.4> <17.28> <20.74> <24.88> mathb12
}{} \DeclareSymbolFont{mathb}{U}{mathb}{m}{n}
\DeclareMathSymbol{\llcurly}{3}{mathb}{"CE}
\DeclareMathSymbol{\ggcurly}{3}{mathb}{"CF}
\DeclareFontFamily{U}{matha}{\hyphenchar\font45}
\DeclareFontShape{U}{matha}{m}{n}{
      <5> <6> <7> <8> <9> <10> gen * matha
      <10.95> matha10 <12> <14.4> <17.28> <20.74> <24.88> matha12
      }{}
\DeclareSymbolFont{matha}{U}{matha}{m}{n}
\DeclareMathSymbol{\curlywedge} {2}{matha}{"4E}
\DeclareMathSymbol{\curlyvee} {2}{matha}{"4F}

\usepackage[all]{xy}
\usepackage{tikz}

\title{Local and global stability of equilibria for a class of chemical reaction networks}

\author{Pete Donnell\footnotemark[1] 
\and Murad Banaji\footnotemark[1] \footnotemark[2] }

\begin{document}

\maketitle

\renewcommand{\thefootnote}{\fnsymbol{footnote}}

\footnotetext[1]{Department of Mathematics, University of Portsmouth, Lion Gate Building, Lion Terrace, Portsmouth, Hampshire PO1 3HF, UK.}
\footnotetext[2] {corresponding author: {\tt murad.banaji@port.ac.uk}}

\begin{abstract}
A class of chemical reaction networks is described with the property that each positive equilibrium is locally asymptotically stable relative to its stoichiometry class, an invariant subspace on which it lies. The reaction systems treated are characterised primarily by the existence of a certain factorisation of their stoichiometric matrix, and strong connectedness of an associated graph. Only very mild assumptions are made about the rates of reactions, and in particular, mass action kinetics are not assumed. In many cases, local asymptotic stability can be extended to global asymptotic stability of each positive equilibrium relative to its stoichiometry class. The results are proved via the construction of Lyapunov functions whose existence follows from the fact that the reaction networks define monotone dynamical systems with increasing integrals. 
\end{abstract}



\section{Introduction}

Systems of chemical reactions can give rise to dynamical systems of various kinds (discrete or continuous time, discrete or continuous state, stochastic or deterministic, for example) and displaying a variety of behaviours \cite{erdi}. Perhaps most widely studied are models whose evolution is naturally described by systems of ordinary differential equations, namely deterministic, continuous time, spatially homogeneous models where chemical concentrations take nonnegative real values. A broad question of interest is when models of some {\bf chemical reaction network} (CRN) can be shown to allow, or forbid, certain behaviours for all reasonable choices of chemical reaction rates (kinetics). The attempt to make claims about the behaviour of CRNs which are to some degree independent of choices of kinetics is often termed {\bf chemical reaction network theory} (CRNT). Since the pioneering work of Feinberg \cite{feinberg0} and Horn and Jackson \cite{hornjackson}, there has been considerable progress in various directions, including the discovery of structural features of networks associated with multistationarity, oscillation, and the persistence of solutions. 

Much, though not all, work in CRNT has focussed on reaction networks with mass action kinetics, but unknown rate constants, namely on particular polynomial differential equations with unknown parameters. Here we construct a class of CRNs which can be proved to have strong convergence properties with weaker assumptions on the kinetics. First note that the evolution of a CRN quite naturally takes place on certain invariant convex sets termed stoichiometry classes (to be defined below). The basic convergence properties of the networks we describe are:

\begin{enumerate}
\item No more than one positive equilibrium on each stoichiometry class, and local asymptotic stability of each positive equilibrium on its stoichiometry class.
\item Under additional assumptions, global asymptotic stability of a unique positive equilibrium on each nontrivial stoichiometry class.
\end{enumerate}

The precise meaning of these statements will be clarified below. The results will be proved using the theory of monotone dynamical systems \cite{hirschsmith,halsmith}. There is a considerable intersection between this theory and the study of CRNs, reflecting the fact that CRNs fairly frequently give rise to order-preserving dynamical systems -- see \cite{volpert,kunzesiegelmathchem,minchevasiegel,leenheer,banajidynsys,angelileenheersontag,angelisontagfutile,banajimierczynski} for example. The key geometrical insights for the results presented here come from the results on monotone systems with increasing integrals in \cite{mierczynski}, generalised in \cite{banajiangeli,banajiangelierratum}. 

\section{Statement of the results}

The local and global results summarised above will be stated precisely as Theorems~\ref{mainthm0}~and~\ref{mainthm} below after some terminology and notation are introduced. Define $\mathbb{R}^n_{\geq 0}$ to be the nonnegative orthant in $\mathbb{R}^n$, i.e.
\[
\mathbb{R}^n_{\geq 0} = \{x \in \mathbb{R}^n\,:\, x_i \geq 0\,\,\,\mbox{for}\,\,\, i = 1, \ldots, n\}\,.
\]
Similarly 
\[
\mathbb{R}^n_{\leq 0} = \{x \in \mathbb{R}^n\,:\, x_i \leq 0\,\,\,\mbox{for}\,\,\, i = 1, \ldots, n\}\,.
\]
A vector in $\mathbb{R}^n_{\geq 0}$ will be referred to as {\bf nonnegative}, while one in $\mathrm{int}(\mathbb{R}^n_{\geq 0})$ will be termed {\bf positive}. As chemical concentrations are necessarily nonnegative, $\mathbb{R}^n_{\geq 0}$ is the natural state-space for ODE models of systems of chemical reactions. \\

We will be considering dynamical systems of the form:
\begin{equation}
\label{eqmain}
\dot x = \Gamma v(x)\,,
\end{equation}
with the following assumptions:
\begin{enumerate}
\item[A1.] $x \in \mathbb{R}^n_{\geq 0}$, $\Gamma$ is an $n \times m$ matrix, $v:U \to \mathbb{R}^m$ is $C^1$ and is defined on some open neighbourhood $U$ of $\mathbb{R}^n_{\geq 0}$. 
\item[A2.] The reaction rates $v(x)$ satisfy conditions K1, K2 and K3 listed in Appendix~\ref{appkinetic}.
\item[A3.] $\Gamma = \Lambda\Theta$, where
\begin{enumerate}
\item[i.] $\Lambda$ is an $n \times r$ matrix with each row containing exactly one nonzero entry, and no column of zeros. 
\item[ii.] $\Theta$ is an $r \times m$ matrix such that $\Theta_{ij}\Theta_{kj} \leq 0$ for $i \not = k$ and $\mathrm{ker}(\Theta^T)$ is one dimensional and includes a positive vector.
\end{enumerate}
\item[A4.] The DSR graph for the system at each $x \in \mathrm{int}(\mathbb{R}^n_{\geq 0})$ is strongly connected.
\end{enumerate}
\vspace{0.3cm}

{\bf Remark on condition A1.} $x$ describes the concentrations of a set of $n$ chemical species involved in $m$ chemical reactions. The matrix $\Gamma$ is termed the {\bf stoichiometric matrix} of the system and $\Gamma_{ij}$ defines the net production/consumption of species $i$ by reaction $j$. It is convenient, though not necessary, to assume that the vector field is defined on an open set containing the nonnegative orthant, in order to avoid technicalities when discussing its derivative. Note that in this paper we follow the convention that a reversible reaction is treated as a single process (rather than as a pair of irreversible reactions), contributing one column to $\Gamma$ and one reaction rate in $v$. In the case of an irreversible reaction we adopt the convention that reactants occur on the left (and are hence associated with negative entries in $\Gamma$), while products occur on the right (and are associated with positive entries in $\Gamma$). In the case of reversible reactions, the choice of ``left'' and ``right'' is arbitrary: altering this choice for the $j$th reaction re-signs the $j$th column of $\Gamma$ and the $j$th reaction rate $v_j$. All results in this paper are independent of this choice.

{\bf Remark on condition A2.} This is a weak assumption on the kinetics. It can be crudely summarised via the statements: ``reactions need all their reactants to proceed'' and ``provided all reactants are present, increasing a concentration speeds up a reaction''. The assumption has been discussed and illustrated previously in \cite{angelileenheersontag,banajimierczynski} for example. The assumption implies, amongst other things, that the nonnegative orthant is positively invariant (Lemma~10 in \cite{banajimierczynski}). With condition A1, this guarantees that (\ref{eqmain}) defines a local semiflow $\phi$ on $\mathbb{R}^n_{\geq 0}$.

{\bf Remark on condition A3.} Condition~A3(i) implies that $\Lambda$ has rank $r$, and condition~A3(ii) implies that each column of $\Theta$ contains exactly one negative entry and exactly one positive entry. The implications of condition A3 will be explored in detail later during proof of the results. In brief, the first factor $\Lambda$ will define a closed, convex and pointed cone in $\mathbb{R}^n$ and the system will be shown to be order-preserving with respect to the order defined by this cone. The nature of this order-cone and its relationship with invariant subspaces of the system will imply the local and global convergence results below. Several lemmas which precede the proofs of these results require weaker conditions on $\Lambda$ and $\Theta$ than A3(i) and A3(ii); condition A3 can thus be seen as the intersection of the various conditions needed in these preliminary lemmas. We will comment in the concluding section on how to identify matrices which admit a factorisation as in condition A3.

{\bf Remark on condition A4.} The DSR graph associated with a CRN \cite{banajicraciun2} is a signed, labelled, bipartite, multidigraph, with relationships to other well-known objects such as Petri nets. It is strongly connected  if there is a (directed) path from each vertex to each other vertex. We need only the following reduced construction here. Given an $n \times m$ matrix $A$ and an $m \times n$ matrix $B$, the (reduced) DSR graph $G_{A, B}$ is defined as follows: it is a bipartite digraph on $n+m$ vertices, $S_1, \ldots, S_n$ and $R_1, \ldots, R_m$ with arc $R_jS_i$ if, and only if, $A_{ij}\not = 0$ and arc $S_iR_j$ if, and only if, $B_{ji} \not = 0$. If desired, the arcs may be given the signs of the associated entries in $A$ and $B$. For System (\ref{eqmain}), at each $x \in \mathbb{R}^n_{\geq 0}$ we can define $G(x) = G_{\Gamma, -Dv(x)}$. Under condition A2, $G(x)$ is constant on $\mathrm{int}(\mathbb{R}^n_{\geq 0})$, and in fact on each elementary face of $\mathbb{R}^n_{\geq 0}$ (defined below).\\

{\bf Notation and terminology.} Given any $n \times k$ matrix $A$, and $x, y \in \mathbb{R}^n$ we will write $x \sim^A y$ for $x - y \in \mathrm{Im}\,A$. Clearly $\sim^A$ is an equivalence relation on $\mathbb{R}^n$. Given any such matrix $A$ and any $x \in \mathbb{R}^n_{\geq 0}$, define
\[
\mathcal{C}_{A, x} \equiv (x + \mathrm{Im}(A)) \cap \mathbb{R}^n_{\geq 0} = \{y \in \mathbb{R}^n_{\geq 0}\,:\, y \sim^A x\}.
\]
In the study of chemical reactions, $\mathcal{C}_{\Gamma, x}$ is termed the {\bf stoichiometry class} of $x$ (also known as the ``stoichiometric compatibility class'' of $x$); for system (\ref{eqmain}) satisfying assumptions A1--A4, $\mathcal{C}_{\Lambda, x}$ will be termed the {\bf $\Lambda$-class} of $x$. Since A3 implies that $\mathrm{Im}(\Gamma) \subseteq \mathrm{Im}(\Lambda)$, stoichiometry classes are subsets of $\Lambda$-classes, and clearly both are (forward) invariant under the local semiflow $\phi$. Stoichiometry classes or $\Lambda$-classes intersecting $\mathrm{int}(\mathbb{R}^n_{\geq 0})$ will be termed {\bf nontrivial}. \\

The first result of this paper is the following local claim:

\begin{theorem}
\label{mainthm0}
Suppose that System (\ref{eqmain}) satisfies assumptions A1, A2, A3 and A4. Each equilibrium $e \in \mathrm{int}(\mathbb{R}^n_{\geq 0})$ is the unique equilibrium on its stoichiometry class $\mathcal{C}_{\Gamma, e}$ and is locally asymptotically stable relative to $\mathcal{C}_{\Gamma, e}$.
\end{theorem}\\

This will be proved via construction of a Lyapunov function on a neighbourhood in $\mathcal{C}_{\Gamma,e}$ of any positive equilibrium $e$. In order to extend this result to a global result, we need conditions to ensure that each nontrivial stoichiometry class contains an equilibrium, the Lyapunov function exists on the whole relative interior of each nontrivial stoichiometry class, and moreover that trajectories cannot approach the relative boundary of a nontrivial stoichiometry class. To make clear these notions we need to introduce some additional ideas. 

Given an $n \times r$ matrix $\Lambda$, define the closed, convex cone $K(\Lambda) \subseteq \mathbb{R}^n$ associated with $\Lambda$ as:
\[
K(\Lambda) = \{\Lambda y\,:\, y \in \mathbb{R}^r_{\geq 0}\}\,.
\]
A local semiflow on $\mathbb{R}^n_{\geq 0}$ is {\bf persistent} if 
\[
x \in \mathrm{int}(\mathbb{R}^n_{\geq 0}) \Rightarrow \omega(x) \cap \partial \,\mathbb{R}^n_{\geq 0} = \emptyset
\]
where $\omega(x)$ is the $\omega$-limit set of $x$.

Let $S \subseteq \{1, \ldots, n\}$ and $S^c = \{1, \ldots, n\}\backslash S$. Define 
\[
F_S = \{x \in \mathbb{R}^n\,:\, x_i > 0, i \in S\,\,\, \mbox{and}\,\,\, x_i = 0, i \not \in S\}.
\]
$F_S$ will be referred to as an {\bf elementary face} of $\mathbb{R}^n_{\geq 0}$. (Elementary faces are the relative interiors of the closed faces of $\mathbb{R}^n_{\geq 0}$.) An elementary face other than $\mathrm{int}(\mathbb{R}^n_{\geq 0})$ or $\{0\}$ will be termed nontrivial. An elementary face $F_S$ is {\bf repelling} if, at each $x \in F_S$, there exists $i \in S^c$ such that $\dot x_i > 0$. Quite generally a repelling face of $\mathbb{R}^n_{\geq 0}$ can contain no $\omega$-limit points of a local semiflow on $\mathbb{R}^n_{\geq 0}$ (Lemma~11 in \cite{banajimierczynski}). \\

To the list of conditions A1--A4 we add two further conditions:
\begin{enumerate}
\item[A5.] $K(\Lambda) \cap \mathbb{R}^n_{\leq 0} = \{0\}$.
\item[A6.] 
\begin{enumerate}
\item[i.] All reactions are reversible, or 
\item[ii.] Every nontrivial elementary face of $\mathbb{R}^n_{\geq 0}$ which intersects a nontrivial stoichiometry class of the system is repelling.
\end{enumerate}
\end{enumerate}
\vspace{0.3cm}

{\bf Remark on condition A5.} It will be shown later that condition A5 guarantees that the Lyapunov function constructed via conditions A1--A4 extends to the entire relative interior of each nontrivial stoichiometry class. Given $\Lambda$ of the form in condition~A3(i), it is clear that condition~A5 is not satisfied if and only if $\Lambda$ includes a column in $\mathbb{R}^n_{\leq 0}$, i.e., with no positive entries.

{\bf Remark on condition A6.} It will be shown in Lemma~\ref{lemsiphon} that conditions A1--A3 and A6(i) imply condition A6(ii) which, by the remarks above, implies the following: given any $x \in \mathrm{int}(\mathbb{R}^n_{\geq 0})$ and $y \in \mathcal{C}_{\Gamma, x}$, then $\omega(y) \cap \partial \mathbb{R}^n_{\geq 0} = \emptyset$, namely $\mathcal{C}_{\Gamma, x} \cap \partial \mathbb{R}^n_{\geq 0}$ contains no limit points of the local semiflow $\phi$. Note that this is a stronger conclusion than persistence of $\left.\phi\right|_{\mathcal{C}_{\Gamma, x}}$. The conclusion that for systems of reversible reactions satisfying conditions A1--A3 and A6(i), equilibria on each nontrivial stoichiometry class attract the whole class, including its boundary, is of some interest in itself: invariant elementary faces of $\mathbb{R}^n_{\geq 0}$ do not intersect nontrivial stoichiometry classes at all, thus making persistence a {\em structural} feature of the systems. The occurrence of such structural persistence was already remarked on in \cite{feinberg}; in contrast proving persistence for CRNs with mass action kinetics in the general situation where nontrivial stoichiometry classes may intersect invariant faces of $\mathbb{R}^n_{\geq 0}$ is delicate (\cite{craciunnazarovpantea,panteapersistence} for example) .\\

We have the following global result:

\begin{theorem}
\label{mainthm}
Suppose that System (\ref{eqmain}) satisfies conditions A1, A2, A3, A4, A5 and A6. Then each nontrivial stoichiometry class contains exactly one equilibrium, which lies in $\mathrm{int}(\mathbb{R}^n_{\geq 0})$, and is globally asymptotically stable relative to its stoichiometry class. 
\end{theorem}\\

In physical terms chemical reaction networks satisfying conditions A1--A6 have very simple behaviour: different initial conditions on the same nontrivial stoichiometry class converge to the same positive equilibrium. 

\section{Examples}

The proofs of Theorems~\ref{mainthm0}~and~\ref{mainthm} are somewhat involved, so some examples of their application are provided first. 

\subsection{Example 1}

Consider the system of three reactions involving four chemicals
\begin{equation}
\label{eqdependent}
A \rightleftharpoons B + C, \qquad B \rightleftharpoons D, \qquad C + D \rightleftharpoons A 
\end{equation}
and the associated differential equation (\ref{eqmain}) with assumptions A1, A2. The stoichiometric matrix, $\Gamma$ admits a factorisation $\Gamma = \Lambda\Theta$ as follows: 
\[
\left(\begin{array}{rrr}-1&0&1\\1&-1&0\\1&0&-1\\0&1&-1\end{array}\right) = \left(\begin{array}{rrr}1&\,\,\,0&\,\,\,0\\0&1&0\\-1&0&0\\0&0&1\end{array}\right)\left(\begin{array}{rrr}-1&0&1\\1&-1&0\\0&1&-1\end{array}\right)
\]
Note that $\Lambda$ and $\Theta$ fulfil conditions A3. The DSR graph at each interior point (drawn under assumption A2) is shown in Figure~\ref{SRdependent}, and is clearly strongly connected. So condition A4 is satisfied. By observation, condition A5 holds. Finally, all reactions are assumed to be reversible, so condition A6 holds. By Theorem~\ref{mainthm} each nontrivial stoichiometry class contains exactly one equilibrium which attracts the whole stoichiometry class. 

\begin{figure}[h]
\begin{center}
\begin{minipage}{0.6\textwidth}
\begin{center}
\begin{tikzpicture}[domain=0:4,scale=0.45]
\fill (1,4) circle (4pt);
\node at (4,4) {$B$};
\draw[-, thick] (1.4,4) -- (3.6,4);
\node at (2.5,2.5) {$C$};
\draw[-, thick] (2.2,2.8) -- (1.2, 3.8);
\draw[-, dashed, thick] (2.8,2.2) -- (3.8, 1.2);
\fill (4,1) circle (4pt);
\node at (1,1) {$A$};
\node at (7,1) {$D$};
\fill (7,4) circle (4pt);
\draw[-, thick] (1.5,1) -- (3.6,1);
\draw[-, dashed, thick] (1,1.5) -- (1,3.6);
\draw[-, thick, dashed] (4.4,1) -- (6.5,1);
\draw[-, dashed, thick] (4.5,4) -- (6.6,4);
\draw[-, thick] (7,1.5) -- (7,3.7);
\end{tikzpicture}

\end{center}
\end{minipage}
\end{center}
\caption{\label{SRdependent} The DSR graph for reaction system (\ref{eqdependent}) at any interior point in $\mathbb{R}^n_{\geq 0}$. Edge-labels are omitted. Negative edges are represented as dashed lines, while positive edges are shown with bold lines: only connectedness is important for the results here, and so edge signs are not needed; however including these makes the relationship between the DSR graph and the network of reactions (\ref{eqdependent}) clearer to see. Pairs of antiparallel arcs of the same sign are represented as single undirected edges.}
\end{figure}
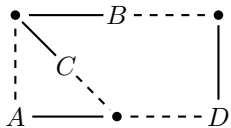

We remark that this example also fulfils the conditions in \cite{angelileenheersontag}, from which global stability follows from a rather different argument. This is not the case for the next example.

\subsection{Example 2} 

Consider the following system of 4 chemical reactions on 5 chemicals:
\begin{equation}
\label{eqdependent1}
A \rightleftharpoons B + C, \qquad B \rightleftharpoons D, \qquad C + D \rightleftharpoons A, \qquad C + E \rightleftharpoons A.
\end{equation}
As before, make assumptions A1, A2 about the kinetics. The stoichiometric matrix, $\Gamma$ admits a factorisation $\Gamma = \Lambda\Theta$ as follows: 
\[
\left(\begin{array}{rrrr}-1&0&1&1\\1&-1&0&0\\1&0&-1&-1\\0&1&-1&0\\0&0&0&-1\end{array}\right) = \left(\begin{array}{rrrr}1&\,\,\,0&\,\,\,0&\,\,\,0\\0&1&0&0\\-1&0&0&0\\0&0&1&0\\0&0&0&1\end{array}\right)\left(\begin{array}{rrrr}-1&0&1&1\\1&-1&0&0\\0&1&-1&0\\0&0&0&-1\end{array}\right).
\]
It is again easy to confirm that condition A3 is satisfied, and that the DSR graph at each interior point under assumption A2 (Figure~\ref{SRdep1}) is strongly connected, so that condition A4 is satisfied. 

\begin{figure}[h]
\begin{center}
\begin{tikzpicture}[domain=0:4,scale=0.45]
\fill (1,4) circle (4pt);
\node at (4,4) {$C$};
\draw[-, thick, dashed] (1.4,4) -- (3.5,4);
\node at (2.5,2.5) {$E$};
\draw[-, thick, dashed] (2.2,2.8) -- (1.2, 3.8);

\fill (4,1) circle (4pt);
\node at (1,1) {$A$};
\node at (7,1) {$B$};
\fill (10,1) circle (4pt);
\fill (7,4) circle (4pt);
\node at (10,4) {$D$};
\draw[-, thick, dashed] (1.5,1) -- (3.6,1);
\draw[-, thick] (1,1.5) -- (1,3.6);
\draw[-, thick] (4.4,1) -- (6.5,1);
\draw[-, thick, dashed] (7.5,1) -- (9.6,1);
\draw[-, thick] (10,1.4) -- (10,3.5);
\draw[-, thick, dashed] (7.4,4) -- (9.5,4);
\draw[-, thick, dashed] (4.5,4) -- (6.6,4);
\draw[-, thick] (4,1.4) -- (4,3.5);

\draw[-, thick] (6.7,4.3) .. controls (1,7) and (-1.5,4.5) .. (0.7,1.3);

\end{tikzpicture}

\end{center}
\caption{\label{SRdep1} The DSR graph for reaction system \ref{eqdependent1} at each interior point of $\mathbb{R}^n_{\geq 0}$. Conventions are as in Figure~\ref{SRdependent}.}
\end{figure}
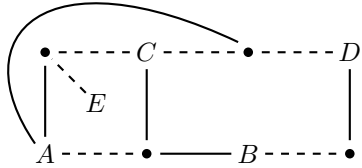

By observation, condition A5 holds. Condition A6 again holds automatically as reactions have been assumed to be reversible. By Theorem~\ref{mainthm} each nontrivial stoichiometry class contains exactly one equilibrium which attracts the whole stoichiometry class. 

To illustrate the case where reactions are not all reversible, we consider one further example. 

\subsection{Example 3}

We consider an example sometimes termed an ``enzymatic futile cycle'' \cite{angelisontagfutile} and whose biological importance is discussed in some detail in \cite{fellmetabolism}. Global stability in this example can be proven in many ways (\cite{angelisontagfutile} for example), but is also an immediate consequence of the results in this paper. 

The enzymatic futile cycle:
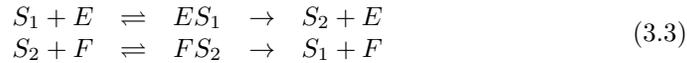
\begin{equation}
\label{reac3}
\begin{array}{ccccc}
S_1 + E & \rightleftharpoons & ES_1 & \rightarrow & S_2 + E\\
S_2 + F & \rightleftharpoons & FS_2 & \rightarrow & S_1 + F
\end{array}
\end{equation}
has stoichiometric matrix which factorises as follows:
\[
\left(\begin{array}{rrrr}-1 & 0 & 0 & 1\\-1 & 1 & 0 & 0\\1 & -1 & 0 & 0\\0 & 1 & -1 & 0\\0 & 0 & -1 & 1\\0 & 0 & 1 & -1\end{array}\right) = \left(\begin{array}{rrrr}1&0&0&0\\0&-1&0&0\\0&1&0&0\\0&0&1&0\\0&0&0&-1\\0&0&0&1\end{array}\right)\left(\begin{array}{rrrr}-1&0&0&1\\1&-1&0&0\\0&1&-1&0\\0&0&1&-1\end{array}\right),
\]
and DSR graph shown in Figure~\ref{DSR3}.
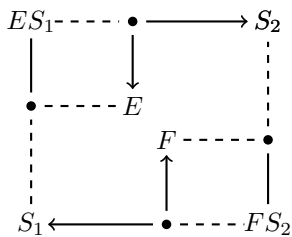
\begin{figure}[h]
\begin{center}
\begin{tikzpicture}[domain=0:4,scale=0.45]

\node at (1,1.5) {$S_1$};
\fill (1,5) circle (4pt);
\fill (5,1.5) circle (4pt);
\node at (1,7.5) {$ES_1$};
\fill (4,7.5) circle (4pt);
\node at (8,7.5) {$S_2$};
\node at (4,5) {$E$};
\node at (5,4) {$F$};
\node at (8,1.5) {$FS_2$};
\fill (8,4) circle (4pt);
\node at (8,7.5) {$S_2$};
\draw[<-, thick] (1.5,1.5) -- (4.6,1.5);
\draw[-, thick, dashed] (5.4,1.5) -- (7.3,1.5);
\draw[-, thick, dashed] (5.5,4) -- (7.6,4);

\draw[-, thick, dashed] (1,2.1) -- (1,4.6);
\draw[-, thick] (1, 5.4) -- (1,7);
\draw[<-, thick] (4, 5.5) -- (4,7.1);

\draw[-, thick, dashed] (1.7,7.5) -- (3.6,7.5);
\draw[->, thick] (4.4,7.5) -- (7.4,7.5);

\draw[-, thick, dashed] (1.4,5) -- (3.5,5);

\draw[-, thick] (8,2.1) -- (8,3.6);
\draw[-, thick, dashed] (8, 4.4) -- (8,6.9);

\draw[->, thick] (5,1.9) -- (5,3.55);

\end{tikzpicture}
\end{center}
\caption{\label{DSR3} The DSR graph for reaction system \ref{reac3} which is clearly strongly connected. Conventions are as in Figure~\ref{SRdependent}. Note that as a consequence of the irreversibility of some reactions, some edges in the DSR graph do not have an oppositely directed partner and so appear as directed.}
\end{figure}

Conditions A1 and A2 are assumed, while conditions A3--A5 are easy to confirm. On the other hand, Condition A6(ii) can be directly checked. The details involve the notions of {\bf siphons} and associated faces of $\mathbb{R}^n_{\geq 0}$, developed in Appendix~\ref{appsiphon}, and are presented in Appendix~\ref{appEx3}. 

\section{Proofs of the results}

\subsection{Summary of the proofs}

The proofs proceed roughly as follows. The key objects of interest are $\Lambda$-classes which are in general invariant $r$-dimensional objects in $\mathbb{R}^n$. Conditions A1--A3 will imply that on each $\Lambda$-class $\phi$ is monotone with respect to the ordering defined by $K(\Lambda)$. Additionally, if condition A4 holds, then $\phi$ is strongly monotone in the relative interior of each nontrivial $\Lambda$-class (in a sense to be made precise below). Condition A3 (ii) allows the construction of an increasing (linear) first integral on each $\Lambda$-class such that stoichiometry classes become the level sets of this function. One implication is that stoichiometry classes are antichains in the ordering defined by $K(\Lambda)$. Via a construction closely related to those in \cite{mierczynski,banajiangeli}, conditions A1--A4 allow the construction of a Lyapunov function on each stoichiometry class in a neighbourhood of any positive equilibrium, strictly increasing along nontrivial trajectories. Adding condition A5 allows this Lyapunov function to be extended to the entire relative interior of the stoichiometry class, while assumption A6 ensures that trajectories on nontrivial stoichiometry classes cannot have $\omega$-limit sets intersecting $\partial \mathbb{R}^n_{\geq 0}$. While some of the arguments are minor modifications of arguments in \cite{banajiangeli}, often they are simpler than the general arguments there, and so the presentation will be mostly self-contained. 

\subsection{Basics on cones and partial orders}

A closed, convex and pointed cone $K \subseteq \mathbb{R}^n$ (i.e., a closed, convex cone additionally satisfying $K \cap (-K) = \{0\}$) will be termed a {\bf CCP cone}. If a CCP cone $K$ has nonempty interior, then it will be termed {\bf proper} \cite{berman}. Define $K^{*}$ to be the {\bf dual cone} to $K$, i.e., $K^* = \{\lambda \in \mathbb{R}^n\,|\, \langle \lambda, y \rangle \geq 0\,\,\mbox{for all}\,\, y \in K\}$.

Given an $n \times r$ matrix $\Lambda$ with rank $r$, then $K(\Lambda)$ is an $r$-dimensional CCP cone in $\mathbb{R}^n$ generated by $r$ extremal vectors (the columns of $\Lambda$). However, we can consider $K(\Lambda)$ to be proper on any coset of $\mathrm{Im}\,\Lambda$ (the affine hull of $K(\Lambda)$) in the following sense: given $z \in \mathbb{R}^r$ and $c \in \mathbb{R}^n$, the map $z \mapsto c + \Lambda z$ is a linear bijection from $\mathbb{R}^r$ to a coset of $\mathrm{Im}\,\Lambda$, which maps the proper cone $\mathbb{R}^r_{\geq 0}$ to $c + K(\Lambda)$. Via this identification, standard results on monotone dynamical systems can be lifted to cosets of $\mathrm{Im}\,\Lambda$. 

The symbols $<,>,\leq, \geq, \ll, \gg$ will refer to the standard partial ordering on $\mathbb{R}^n$. For example, for $a,b \in \mathbb{R}^n$, $a < b$ means $b-a \in \mathbb{R}^n_{\geq 0} \backslash\{0\}$. When the ordering is that defined by some other cone $K$, the alternative symbols $\prec, \succ, \preceq, \succeq, \llcurly, \ggcurly$ will be used. (Where these symbols are used, the cone in question will be clear from the context.) For example, given some CCP cone $K\subseteq \mathbb{R}^n$ and $a, b \in \mathbb{R}^n$, $a \preceq b$ means $b - a \in K$. Normally $a \llcurly b$ means $b-a \in \mathrm{int}(K)$; here the meaning will be extended so that given any CCP cone $K\subseteq \mathbb{R}^n$, $a \llcurly b$ means $b-a \in \mathrm{relint}(K)$.

Given a partial order $\preceq$ defined by a cone $K$, and some set $X$, if $x, y \in X$ implies $x \not \prec y$, we will say that ``$X$ is a $K$-antichain''.

\subsection{$\Lambda$-classes are lattices}

The lemmas to follow will show that the cones considered here define orderings which make each $\Lambda$-class a lattice. Given $a, b \in \mathbb{R}^r$, define $c=a \wedge b$ by $c_i = \min\{a_i, b_i\}$, and $c=a \vee b$ by $c_i = \max\{a_i, b_i\}$. 

\begin{lemma}
\label{latticefull}
Let $\Lambda$ be an $n \times r$ matrix with rank $r$ defining a cone $K(\Lambda) \subseteq \mathbb{R}^n$. Let $c \in \mathbb{R}^n$ be an arbitrary vector. 
\begin{enumerate}
\item Consider $x, y \in \mathbb{R}^r$, $z = x \wedge y$, $x' = c+ \Lambda x$, $y' = c+ \Lambda y$, and $z' = c+ \Lambda z$. Then (i) $z' \preceq x'$ and $z' \preceq y'$ and (ii) Any $b' \in c+\mathrm{Im}(\Lambda)$ satisfying $b' \preceq x'$ and $b' \preceq y'$, also satisfies $b' \preceq z' $. In other words, $z'$ is the infimum of $x', y'$ in the order defined by $K(\Lambda)$.
\item Consider $x, y \in \mathbb{R}^r$, $z = x \vee y$, $x' = c+ \Lambda x$, $y' = c+ \Lambda y$, and $z' = c+ \Lambda z$. Then (i) $z' \succeq x'$ and $z' \succeq y'$ and (ii) Any $b' \in c+\mathrm{Im}(\Lambda)$ satisfying $b' \succeq x'$ and $b' \succeq y'$, also satisfies $b' \succeq z' $. In other words, $z'$ is the supremum of $x', y'$ in the order defined by $K(\Lambda)$.
\end{enumerate}
\end{lemma}
\begin{proof}
Part 1 will proved. The proof of Part 2 is similar. Since $z \leq x$,  $x-z = p \in \mathbb{R}^r_{\geq 0}$. So $x'-z' = (c+\Lambda x)-(c+\Lambda z) = \Lambda p$, i.e., $z' \preceq x'$. Similarly $z' \preceq y'$. 

Consider a vector $b' = c+\Lambda b$ satisfying $b' \preceq x'$ and $b' \preceq y'$. i.e., 
\[
x'-b' = c+\Lambda x - (c+\Lambda b) = \Lambda p_1 \quad \mbox{and} \quad y' - b' = c+\Lambda y - (c + \Lambda b) = \Lambda p_2\,,
\]
where $p_1, p_2 \in \mathbb{R}^r_{\geq 0}$. Multiplying each equation by any matrix $\Lambda '$ such that $\Lambda '\Lambda  = I$ gives $b \leq x$ and $b \leq y$, implying $b \leq x \wedge y \equiv z$, i.e. $z - b = p_3 \in \mathbb{R}^r_{\geq 0}$. So $z' - b' = (c+\Lambda z) - (c+\Lambda b) = \Lambda p_3 \in K(\Lambda)$, i.e. $b'\preceq z'$. 
\end{proof}\\

{\bf Notation.} Given any $\Lambda$ satisfying the assumptions of Lemma~\ref{latticefull}, and $x, y \in \mathbb{R}^n$ with $x \sim^\Lambda y$, we write $x \curlywedge y$ for the infimum of $x, y$ under the order defined by $K(\Lambda)$. Similarly $x \curlyvee y$ refers to the supremum of $x, y$ under the order defined by $K(\Lambda)$. Clearly $(x \curlyvee y) \sim^\Lambda (x \curlywedge y) \sim^\Lambda x \sim^\Lambda y$.\\

The next lemma shows that additional assumptions on $\Lambda$ ensure that the intersection of a coset of $\mathrm{Im}(\Lambda)$ with any closed order interval under the standard ordering (including, for example, $\mathbb{R}^n_{\geq 0}$ itself) is a lattice under the order defined by $K(\Lambda)$. 

\begin{lemma}
\label{lat0}
Let $\Lambda$ be an $n \times r$ matrix of rank $r$, and such that each row contains no more than one positive entry and no more than one negative entry. Given $c \in \mathbb{R}^n$, $x', y' \in c+ \mathrm{Im}(\Lambda)$ and $t \in \mathbb{R}^n$:
\begin{enumerate}
\item If $x', y' \geq t$, then $x' \curlywedge y' \geq t$.
\item If $x', y' \geq t$, then $x' \curlyvee y' \geq t$.
\item If $x', y' \leq t$, then $x' \curlywedge y' \leq t$.
\item If $x', y' \leq t$, then $x' \curlyvee y' \leq t$.
\end{enumerate}
\end{lemma}
\begin{proof}
Choose $x, y \in \mathbb{R}^r$ such that $x'=c+\Lambda x, y'=c+\Lambda y$. If row $i$ of $\Lambda$ contains a positive entry, then define $p(i)$ by $\Lambda_{i, p(i)} > 0$, and let $\alpha_i = \Lambda_{i, p(i)}$; otherwise let $p(i)=1$ and $\alpha_i = 0$. Similarly, if row $i$ of $\Lambda$ contains a negative entry, then define $q(i)$ by $\Lambda_{i, q(i)} < 0$, and let $\beta_i = -\Lambda_{i, q(i)}$; otherwise let $q(i) = 1$ and $\beta_i = 0$. Note that for each $i$, $\alpha_i, \beta_i \geq 0$.  We get
\[
(\Lambda x)_i = \alpha_i x_{p(i)} - \beta_i x_{q(i)} \qquad \mbox{and} \qquad (\Lambda y)_i = \alpha_i y_{p(i)} - \beta_i y_{q(i)}\,. 
\]

We prove Statements~1~and~2. Statements~3~and~4 follow analogously. 

{\bf 1.}  Let $z = x \wedge y$ so that, by Lemma~\ref{latticefull}, $z' = c+ \Lambda z = x' \curlywedge y'$.
\[
z'_i = c_i + (\Lambda z)_i = c_i + \alpha_i \min\{x_{p(i)}, y_{p(i)}\} - \beta_i \min\{x_{q(i)}, y_{q(i)}\}\,. 
\]
So either
\begin{eqnarray*}
z'_i & = & c_i + \alpha_i x_{p(i)} - \beta_i x_{q(i)} = c_i + (\Lambda x)_i \geq t_i, \quad \mbox{or}\\
z'_i & = & c_i + \alpha_i y_{p(i)} - \beta_i y_{q(i)} = c_i + (\Lambda y)_i \geq t_i, \quad \mbox{or}\\
z'_i & = & c_i + \alpha_i x_{p(i)} - \beta_i y_{q(i)} \geq c_i + (\Lambda x)_i \geq t_i\quad\mbox{(because}\quad y_{q(i)} \leq x_{q(i)}), \quad \mbox{or}\\
z'_i & = & c_i + \alpha_i y_{p(i)} - \beta_i x_{q(i)} \geq c_i + (\Lambda y)_i\geq t_i\quad\mbox{(because}\quad x_{q(i)} \leq y_{q(i)}).
\end{eqnarray*}
In every case, $z'_i \geq t_i$, so $z' \geq t$.

{\bf 2.} Let $z = x \vee y$ so that, by Lemma~\ref{latticefull}, $z' = c+ \Lambda z = x' \curlyvee y'$.
\[
z'_i= c_i + (\Lambda z)_i = c_i + \alpha_i \max\{x_{p(i)}, y_{p(i)}\} - \beta_i \max\{x_{q(i)}, y_{q(i)}\}\,. 
\]
So either
\begin{eqnarray*}
z'_i & = & c_i + \alpha_i x_{p(i)} - \beta_i x_{q(i)} = c_i + (\Lambda x)_i \geq t_i, \quad \mbox{or}\\
z'_i & = & c_i + \alpha_i y_{p(i)} - \beta_i y_{q(i)} = c_i + (\Lambda y)_i \geq t_i, \quad \mbox{or}\\
z'_i & = & c_i + \alpha_i x_{p(i)} - \beta_i y_{q(i)} \geq c_i + (\Lambda y)_i \geq t_i\quad\mbox{(because}\quad x_{p(i)} \geq y_{p(i)}), \quad \mbox{or}\\
z'_i & = & c_i + \alpha_i y_{p(i)} - \beta_i x_{q(i)} \geq c_i + (\Lambda x)_i\geq t_i\quad\mbox{(because}\quad y_{p(i)} \geq x_{p(i)}).
\end{eqnarray*}
In every case, $z'_i \geq t_i$, so $z' \geq t$.
\end{proof}\\

\begin{corollary}
\label{lat2}
Let $\Lambda$ be an $n \times r$ matrix with rank $r$, and such that each row contains no more than one positive entry and no more than one negative entry. Then, with the partial order defined by $K(\Lambda)$, $\mathcal{C}_{\Lambda, c}$ is a lattice for each $c \in \mathbb{R}^n_{\geq 0}$. Consequently, with assumption A3, each $\Lambda$-class of (\ref{eqmain}) is a lattice. 
\end{corollary}
\begin{proof}
This follows from Statements~1~and~2 of Lemma~\ref{lat0} with $t$ as the origin. 
\end{proof}

\subsection{Each $\Lambda$-class has an infimum}

We show that including condition A5 ensures the existence of a unique minimal element on each $\Lambda$-class. 

\begin{lemma}
\label{minimal}
Consider any $n \times r$ matrix $\Lambda$ of rank $r$, and such that $K(\Lambda) \cap \mathbb{R}^n_{\leq 0} = \{0\}$. Then given any $c \in \mathbb{R}^n_{\geq 0}$, $\mathcal{C}_{\Lambda, c}$ contains a greatest lower bound in the ordering determined by $K(\Lambda)$.
\end{lemma}
\begin{proof}
Consider any chain (i.e., totally ordered subset) $C \subseteq \mathcal{C}_{\Lambda, c}$. Take any decreasing sequence $(x_i) \subseteq C$, i.e., $x_0 \succ x_1 \succ x_2 \succ \cdots$. We want to show that $(x_i)$ is bounded (and hence bounded below). Since the sequence is arbitrary, this will imply that $C$ is bounded below. Suppose, on the contrary, there exists a subsequence $(x_{i_k})$ with $|x_{i_k}|\to \infty$, implying that $|x_{i_k} - x_0|\to \infty$. Consider the bounded sequence $(x_{i_k} - x_{0})/|x_{i_k}|$, and by passing to a subsequence if necessary, assume that it is convergent, i.e., $(x_{i_k} - x_{0})/|x_{i_k}| \to y$. But $(x_{i_k} - x_{0})/|x_{i_k}| = x_{i_k}/|x_{i_k}| - x_0/|x_{i_k}|$, and since $x_{i_k}/|x_{i_k}|$ is nonnegative and of magnitude $1$, and $x_0/|x_{i_k}| \to 0$, it follows that $y$ is nonnegative and of magnitude $1$. On the other hand, $(x_{i_k} - x_{0})/|x_{i_k}| \in -K(\Lambda)$, and since $-K(\Lambda)$ is closed, $y \in -K(\Lambda)$. This implies that $-K(\Lambda)$ includes a nonzero, nonnegative vector, contradicting the assumptions of the theorem. Thus every chain has a lower bound. By Zorn's Lemma, $\mathcal{C}_{\Lambda, c}$ contains a minimal element. Since $\mathcal{C}_{\Lambda, c}$ is a lattice, this minimal element is unique. 
\end{proof}

\subsection{With reversibility, $\omega$-limit sets of nontrivial initial conditions lie in $\mathrm{int}(\mathbb{R}^n_{\geq 0})$}

The next claim is that if all reactions are reversible, no limit sets of System (\ref{eqmain}) satisfying conditions A1--A3 on a nontrivial stoichiometry class can intersect $\partial \mathbb{R}^n_{\geq 0}$. In other words, conditions A1--A3 and A6(i) imply condition A6(ii): 
\begin{lemma}
\label{lemsiphon}
Consider System (\ref{eqmain}) satisfying conditions A1--A3. Suppose additionally that all reactions are reversible, and let $S$ be some nonempty, proper subset of $\{1, \ldots, n\}$. If for some $c \gg 0$, $\mathcal{C}_{\Gamma,c}$ intersects $\overline{F}_S$, then $F_S$ is repelling. 
\end{lemma}
\vspace{0.3cm}

This claim involves a somewhat lengthy digression, and so is proved and illustrated in Appendix~\ref{appsiphon}. Note that the assumption that chemical reactions are reversible means, mathematically, that the rate function $v(x)$ fulfils conditions K1 and K3 in Appendix~\ref{appkinetic}. 

\subsection{An increasing first integral, stoichiometry classes are antichains}

In this subsection, make only the following assumptions on $\Gamma$, implied by assumption A3. Assume that $\Gamma = \Lambda \Theta$ where
\begin{enumerate}
\item[C1.] $\Lambda$ is an $n \times r$ matrix with rank $r$.
\item[C2.] $\Theta$ is an $r \times m$ matrix, $\mathrm{ker}(\Theta^T)$ is one dimensional, and there is a unit vector $y_\Theta \in \mathrm{ker}(\Theta^T)$ satisfying $y_\Theta \gg 0$. 
\end{enumerate}
Since $\Lambda^T$ has rank $r$, it defines a surjective map from $\mathbb{R}^n \to \mathbb{R}^r$, and so we can choose and fix a vector $p_\Theta$ such that $\Lambda^T p_\Theta = y_\Theta$. Note that $p_\Theta \in \mathrm{int}(K(\Lambda)^*)$, the dual cone to $K(\Lambda)$, since for any $z > 0$, $p_\Theta^T\Lambda z = y_\Theta^T z > 0$. Define the linear scalar function $H:\mathbb{R}^n \to \mathbb{R}$ by $H(x) = p_\Theta^Tx$. The next lemma shows that $H$ is increasing with respect to the order defined by $K(\Lambda)$. 
\begin{lemma}
\label{inc}
Consider $x, y \in \mathbb{R}^n$ such that $y \succ x$. Then $H(y) > H(x)$. 
\end{lemma}
\begin{proof}
Note that $y = x + \Lambda z$ where $z > 0$. Then 
\[
H(y)-H(x) = p_\Theta^Ty - p_\Theta^Tx = p_\Theta^T(\Lambda z) = y^T_\Theta z > 0,
\]
where the last inequality follows because $z > 0$, and $y_\Theta \gg 0$. 
\end{proof}\\

Next, we prove that restricting attention to a $\Lambda$-class, stoichiometry classes are precisely the level sets of $H$. 

\begin{lemma}
\label{Scchar}
For any $x \in \mathbb{R}^n_{\geq 0}$, $\mathcal{C}_{\Gamma,x} = \{y \sim^\Lambda x \,|\, H(y) = H(x)\}$.
\end{lemma}
\begin{proof}
From the definitions, $\mathcal{C}_{\Gamma,x} \subseteq \mathcal{C}_{\Lambda,x}$. Note that $p_\Theta^T\Gamma = p_\Theta^T\Lambda\Theta = y_\Theta^T\Theta = 0$. So given $y \sim^\Gamma x$, and writing $y = x + \Gamma y'$, 
\[
H(y) = p_\Theta^Ty = p_\Theta^T(x+\Gamma y') = p_\Theta^Tx = H(x)\,.
\]
So $\mathcal{C}_{\Gamma,x} \subseteq \{y \sim^\Lambda x \,|\, H(y) = H(x)\}$. 

On the other hand, consider any $y \sim^\Lambda x$ such that $H(y) = H(x)$. Write $y = x + \Lambda y'$ and note that 
\[
0 = H(y) - H(x) = p_\Theta^T(y-x) = p_\Theta^T\Lambda y' = y_\Theta^T y'. 
\]
Since $y_\Theta$ spans $\mathrm{ker}(\Theta^T)$ and $\mathrm{Im}(\Theta) = [\mathrm{ker}(\Theta^T)]^{\perp}$, $y' \in \mathrm{Im}(\Theta)$. So $y' = \Theta y''$ for some $y''$, i.e., $y = x + \Lambda \Theta y'' = x + \Gamma y''$. Thus $y \sim^\Gamma x$. So $\mathcal{C}_{\Gamma,x} \supseteq \{y \sim^\Lambda x \,|\, H(y) = H(x)\}$, proving the claim. 
\end{proof}\\

\begin{corollary}
\label{unord1}
Each stoichiometry class of System (\ref{eqmain}) with assumption A3 is a $K(\Lambda)$-antichain.
\end{corollary}
\begin{proof}
Note that condition A3 implies conditions C1 and C2. By Lemma~\ref{Scchar}, $x \sim^\Gamma y$ implies $H(x) = H(y)$. On the other hand, by Lemma~\ref{inc}, $H(x) = H(y)$ implies $x \not \prec y$, and it follows that $\mathcal{C}_{\Gamma,y}$ is a $K(\Lambda)$-antichain. 
\end{proof}\\

{\bf Notation.} Given the characterisation in Lemma~\ref{Scchar}, when we restrict attention to some $\Lambda$-class $\mathcal{C}_{\Lambda, c}$, we can refer to the stoichiometry class in $\mathcal{C}_{\Lambda, c}$ on which $H(\cdot)$ takes the value $h$ as $\mathcal{C}_{\Lambda, c}^h$, i.e. define $\mathcal{C}_{\Lambda, c}^h = \{y \in \mathcal{C}_{\Lambda, c}\,|\, H(y) = h\}$.

\subsection{$K$-quasipositivity}

{\bf Notation for matrices.} Given any matrix $M$, we refer to the $k$th column of $M$ as $M_k$ and the $k$th row of $M$ as $M^k$. It is convenient to phrase results in terms of ``qualitative classes'' of matrices and related ideas. Given a matrix $M$, the matrix-sets $\mathcal{Q}(M)$, $\mathcal{Q}_0(M)$ and $\mathcal{Q}_1(M)$ are defined in Appendix~\ref{appqualclass}.\\

{\bf Terminology.} Given a CCP cone $K \subseteq \mathbb{R}^n$, and an $n \times n$ matrix $J$, we say that $J$ is {\bf $K$-quasipositive} if there exists $\alpha \in \mathbb{R}$ such that $J + \alpha I \colon K \to K$. If in fact there exists $\alpha \in \mathbb{R}$ such that $J + \alpha I \colon K \to K$ and $J + \alpha I$ is $K$-irreducible (namely, if $F$ is a closed face of $K$ and $J + \alpha I \colon F \to F$, then either $F=\{0\}$ or $F=K$), then we say that $J$ is {\bf strictly $K$-quasipositive}.\\

The following lemma on $K$-quasipositivity of a special class of rank $1$ matrices appears as Theorems~5.3~and~5.4 in \cite{banajidynsys}. 

\begin{lemma}
\label{suff2}
Consider a vector $\gamma$, any vector $v \in \mathcal{Q}_0(-\gamma)$ and a CCP cone $K$ with extremals $\{y_i\}$. Define two conditions as follows:
\begin{itemize}
\item[A.] $\gamma = ry_k$ for some $k$, and either (i) $r > 0$ and $y_j \in \mathcal{Q}_1(-\gamma)$ for all $j \not = k$ or (ii) $r < 0$ and $y_j \in \mathcal{Q}_1(\gamma)$ for all $j \not = k$.
\item[B.] There exist $r_1, r_2 > 0$ and $y_i \in \mathcal{Q}_1(\gamma)\backslash\mathcal{Q}_1(-\gamma)$, $y_j \in \mathcal{Q}_1(-\gamma)\backslash\mathcal{Q}_1(\gamma)$ such that $\gamma = r_1y_i - r_2y_j$. Moreover, $y_k \in (\mathcal{Q}_1(\gamma) \cap \mathcal{Q}_1(-\gamma))$ for $k \not \in \{i,j\}$. 
\end{itemize}
If either $\gamma = 0$, or A or B holds then $\gamma v^T$ is $K$-quasipositive. 
\end{lemma}

\begin{proof}
The case $\gamma = 0$ is trivial. For the remaining cases, see \cite{banajidynsys}. Although cones were assumed in that reference to be proper, the proofs are straightforward calculations which apply for any CCP cone.
\end{proof}\\

An immediate corollary is: 
\begin{corollary}
\label{suff3}
Consider a matrix $\Gamma$, any matrix $V \in \mathcal{Q}_0(-\Gamma^T)$ and some CCP cone $K$. If each nonzero column of $\Gamma$ satisfies either condition A or condition B of Lemma~\ref{suff2}, then $\Gamma V$ is $K$-quasipositive. 
\end{corollary}
\begin{proof}
$\Gamma V = \sum_i\Gamma_iV^i$, and by Theorem~\ref{suff2} for each $i$ there exists $\alpha_i$ such that $\Gamma_iV^i + \alpha_i I:K \to K$. Clearly, defining $\alpha = \sum_i \alpha_i$, $\Gamma V + \alpha I:K \to K$. 
\end{proof}\\

This leads to:

\begin{lemma}
\label{quasipos}
Let $\Lambda$ be an $n \times r$ matrix with rank $r$, and no more than one nonzero entry in each row. Let $\Theta$ be an $r \times m$ matrix such that each column of $\Theta$ contains no more than one positive entry and no more than one negative entry. Let $\Gamma = \Lambda\Theta$, $V \in \mathcal{Q}_0(-\Gamma^T)$. Then $\Gamma V$ is $K(\Lambda)$-quasipositive. 
\end{lemma}
\begin{proof}
We need to show that for each $i$, $\Gamma_i$ fulfils the conditions of Lemma~\ref{suff2}. The trivial case $\Theta_i = 0$ implies $\Gamma_i=0$. The reader can easily confirm that if $\Theta_i$ contains a single nonzero entry, then $\Gamma_i$ satisfies condition A, and if $\Theta_i$ contains a positive entry and a negative entry, then $\Gamma_i$ satisfies condition B. 
\end{proof}\\

\begin{corollary}
\label{corqp}
Consider System (\ref{eqmain}) with assumptions A1--A3. At each $x \in \mathbb{R}^n_{\geq 0}$, the Jacobian matrix $\Gamma Dv(x)$ is $K(\Lambda)$-quasipositive. 
\end{corollary}
\begin{proof}
Assumption A2 implies that $Dv(x)\in \mathcal{Q}_0(-\Gamma^T)$ at each $x \in \mathbb{R}^n_{\geq 0}$. Assumption A3 implies the assumptions of Lemma~\ref{quasipos}, which now gives the result.
\end{proof}

\subsection{Strict $K$-quasipositivity}

The aim in this section is to infer the following for System (\ref{eqmain}) satisfying conditions A1--A3: at any point $x$ where the DSR graph $G(x)$ is strongly connected, the Jacobian $J(x)$ is strictly $K(\Lambda)$-quasipositive (i.e., there exists $\alpha \in \mathbb{R}$ such that $J(x) + \alpha I\colon K(\Lambda) \to K(\Lambda)$ and $J(x) + \alpha I$ is $K(\Lambda)$-irreducible). The following is an immediate consequence of Theorem~1 in \cite{banajiburbanks}.

\begin{lemma}
\label{lemDSR}
Let $K \subseteq \mathbb{R}^n$ be a CCP cone, $A$ an $n \times m$ matrix, and $B'$ an $m \times n$ matrix. Suppose that $\mathrm{Im}\,A \not \subseteq \mathrm{span}\,F$ for any nontrivial face $F$ of $K$, and that $AB$ is $K$-quasipositive for each $B \in \mathcal{Q}_0(B')$. Then whenever the DSR graph $G_{A, B}$ is strongly connected, $AB$ is strictly $K$-quasipositive.
\end{lemma}\\

We wish to apply this result with $A = \Gamma$, $B' = -\Gamma^T$, and $K = K(\Lambda)$. We have already seen that $AB$ is $K$-quasipositive for all $B\in \mathcal{Q}_0(-\Gamma^T)$, and so we need:

\begin{lemma}
\label{lemnotinspan}
Consider an $n \times r$ matrix $\Lambda$ with rank $r$ and an $r \times m$ matrix $\Theta$ with no row of zeros. Let $\Gamma = \Lambda\Theta$. Then $\mathrm{Im}\,\Gamma \not \subseteq \mathrm{span}\,F$ for any nontrivial face $F\subseteq K(\Lambda)$.
\end{lemma}
\begin{proof}
Since $\mathrm{rank}\,\Lambda = r$, a vector $z \in \mathrm{Im}\,\Lambda$ has a unique representation $z = \sum \alpha_i\Lambda_i$. Consider some nontrivial face $F$ of $K(\Lambda)$, and choose some $k$ such that $\Lambda_k \not \in \overline F$. Since $\Theta$ has no row of zeros, choose some $i(k)$ such that $\Theta_{k,i(k)} = \alpha \not = 0$. Define $y = \Gamma \hat e_{i(k)} = \Lambda\Theta \hat e_{i(k)} = \alpha \Lambda_k + \cdots$. Since this representation of $y$ is unique, clearly $y \not \in \mathrm{span}\,F$ and so $\mathrm{Im}\,\Gamma \not \subseteq \mathrm{span}\,F$. 
\end{proof}\\

We can deduce that:

\begin{corollary}
\label{corstrict}
Consider System (\ref{eqmain}) with assumptions A1--A3. Assume that at some $x \in \mathbb{R}^n_{\geq 0}$, the DSR graph $G_{\Gamma, Dv(x)}$ is strongly connected. Then the Jacobian matrix $\Gamma Dv(x)$ is strictly $K(\Lambda)$-quasipositive. 
\end{corollary}
\begin{proof}
Condition A3(ii) implies that $\Theta$ has no row of zeros, for if $\Theta_{ij} = 0$ for some $i$ and all $j$, and $z$ is a positive vector in $\mathrm{ker}\,\Theta^T$ which exists by assumption, then $\{z, \hat e_i\}$ are two linearly independent vectors in $\mathrm{ker}\,\Theta^T$. So by Lemma~\ref{lemnotinspan}, $\mathrm{Im}\,\Gamma \not \subseteq \mathrm{span}\,F$ for any nontrivial face $F\subseteq K(\Lambda)$. Since, by Lemma~\ref{quasipos}, $\Gamma Dv(x)$ is $K$-quasipositive for all $Dv(x)\in \mathcal{Q}_0(-\Gamma^T)$, Lemma~\ref{lemDSR} now implies that $\Gamma Dv(x)$ is strictly $K(\Lambda)$-quasipositive whenever $G_{\Gamma, Dv(x)}$ is strongly connected. 
\end{proof}\\

It follows immediately that:
\begin{corollary}
\label{corstrict1}
Consider System (\ref{eqmain}) with assumptions A1--A4. The Jacobian matrix $\Gamma Dv(x)$ is strictly $K(\Lambda)$-quasipositive at each $x \in \mathrm{int}(\mathbb{R}^n_{\geq 0})$. 
\end{corollary}

\subsection{Monotonicity and strong monotonicity}

We have shown in Corollaries~\ref{corqp}~and~\ref{corstrict1} that conditions A1--A4 imply that $\Gamma Dv$ is $K(\Lambda)$-quasipositive on all of $\mathbb{R}^n_{\geq 0}$ and strictly $K(\Lambda)$-quasipositive on $\mathrm{int}(\mathbb{R}^n_{\geq 0})$. The implications in terms of monotonicity and strong monotonicity of the local semiflow restricted to $\Lambda$-classes are discussed briefly. \\

{\bf Notation.} Given $a, b \in \mathbb{R}^n$, define the closed segment $[a, b] = \{\lambda a + (1-\lambda)b\,:\,\lambda \in [0,1]\}$. Open and semi-open segments $(a, b)$, $(a, b]$ and $[a, b)$ are similarly defined.\\

The following is a consequence of results in \cite{hirschsmith} (see also \cite{banajimierczynski}):
\begin{lemma}
\label{hslem}
Consider a proper cone $K \subseteq \mathbb{R}^r$, some open set $U\subseteq \mathbb{R}^r$ , and a $C^1$ vector field $f \colon U \to\mathbb{R}^r$ with Jacobian matrix $Df$. Let $X \subseteq U$ be some convex domain, positively invariant under the local flow $\phi_U$ defined by $f$, and let $\phi$ be the induced local semiflow on $X$. Assume that $Df$ is $K$-quasipositive in $X$. Consider some $x, y \in X$ with $x \preceq y$. Then $\phi_t(x) \preceq \phi_t(y)$ for each $t > 0$ such that $\phi_t(x), \phi_t(y)$ exist. If $x \prec y$ and there exists $z \in [x,y]$ such that $Df(z)$ is strictly $K$-quasipositive, then $\phi_t(x) \llcurly \phi_t(y)$ for each $t > 0$ such that $\phi_t(x), \phi_t(y)$ exist.
\end{lemma}\\

Our interest is in $\Lambda$-classes. Recalling that $\Lambda$ can be regarded as a bijection from $\mathbb{R}^r$ to $\mathrm{Im}(\Lambda)$, mapping $\mathbb{R}^r_{\geq 0}$ to $K(\Lambda)$, Lemma~\ref{hslem}, Corollaries~\ref{corqp}~and~\ref{corstrict1}, together imply:

\begin{corollary}
\label{hscor}
Consider System (\ref{eqmain}) with assumptions A1--A4 and $x, y \in \mathbb{R}^n_{\geq 0}$ with $x \preceq y$. Then $\phi_t(x) \preceq \phi_t(y)$ for each $t > 0$ such that $\phi_t(x), \phi_t(y)$ exist. If $x \prec y$ and at least one of $x, y$ is in $\mathrm{int}(\mathbb{R}^n_{\geq 0})$, then $\phi_t(x) \llcurly \phi_t(y)$ (namely, $\phi_t(y) - \phi_t(x) \in \mathrm{relint}\,K(\Lambda)$) for each $t > 0$ such that $\phi_t(x), \phi_t(y)$ exist.
\end{corollary}

\subsection{Structure of the equilibrium set}

Define $E \subseteq \mathbb{R}^n_{\geq 0}$ to be the equilibrium set of (\ref{eqmain}), i.e. 
\[
E = \{x \in \mathbb{R}^n_{\geq 0}\,:\, \Gamma v(x) = 0\}. 
\]

\begin{lemma}
\label{lemord0}
Consider System (\ref{eqmain}) with assumptions A1--A4 and two distinct equilibria $x, y$ with $x \sim^\Lambda y$ and at least one of $x, y \in \mathrm{int}(\mathbb{R}^n_{\geq 0})$. Then either $x \llcurly y$ or $y \llcurly x$. Consequently no stoichiometry class with an equilibrium in $\mathrm{int}(\mathbb{R}^n_{\geq 0})$ contains more than one equilibrium.
\end{lemma}
\begin{proof}
Assume, by relabelling $x$ and $y$ if necessary, that $x \not \succ y$. 
\begin{enumerate}
\item[(i)] Suppose $x \prec y$, but $x \not \llcurly y$. Since at least one of $x$ or $y$ lies in $\mathrm{int}(\mathbb{R}^n_{\geq 0})$, the line segment $[x,y]$ certainly intersects $\mathrm{int}(\mathbb{R}^n_{\geq 0})$. Corollary~\ref{hscor} then implies that for $t > 0$, $x = \phi_t(x) \llcurly \phi_t(y) = y$, a contradiction. 

\item[(ii)] Now suppose that $x$ and $y$ are not order related. Then $z \equiv x \curlywedge y$ is different from $x$ and $y$. By monotonicity $\phi_t(z) \preceq \phi_t(x) = x$ and $\phi_t(z) \preceq \phi_t(y) = y$. Since $z \equiv x \curlywedge y$, $\phi_t(z) \preceq z$. But $\phi_t(z) \in \mathcal{C}_{\Gamma,z}$ by invariance of $\mathcal{C}_{\Gamma,z}$, and $\mathcal{C}_{\Gamma,z}$ is an antichain (Corollary~\ref{unord1}), so $z \in E$. If $y \in \mathrm{int}(\mathbb{R}^n_{\geq 0})$ (resp. $x \in \mathrm{int}(\mathbb{R}^n_{\geq 0})$), applying the argument in part (i) to $z$ and $y$ (resp. $z$ and $x$) gives a contradiction.
\end{enumerate}

Since stoichiometry classes are subsets of $\Lambda$-classes and are antichains, it follows immediately that no stoichiometry class with an equilibrium in $\mathrm{int}(\mathbb{R}^n_{\geq 0})$ contains more than one equilibrium.
\end{proof}\\

{\bf Remark.} Strong ordering of equilibria followed from a considerably more involved argument in \cite{banajiangeli}. Here the fact that $\Lambda$-classes are lattices makes the conclusion simple. \\

Given $x \sim^\Lambda y$, define 
\[
P(x, y) = ((x+K(\Lambda)) \cup (x - K(\Lambda))) \cap \mathcal{C}_{\Gamma,y}\,.
\]
Using the characterisation in Lemma~\ref{Scchar}, we can alternatively write
\[
P(x, y) = ((x+K(\Lambda)) \cup (x - K(\Lambda))) \cap \{w \sim^\Lambda x \,:\, H(w) = H(y)\}.
\]

\begin{lemma}
\label{foliate}
Assume that matrices $\Lambda$, $\Theta$ and $\Gamma$ satisfy Condition A3. Then $P(x, y)$ is nonempty for any $y \sim^\Lambda x$. 
\end{lemma}
\begin{proof}
If $H(y) = H(x)$, then by Lemma~\ref{Scchar}, $x \in \mathcal{C}_{\Gamma,y}$ and we are done. Assume that $H(y) > H(x)$ (resp. $H(y) < H(x)$). By the lattice property of $\mathcal{C}_{\Lambda, x}$ (Corollary~\ref{lat2}), $z = x \curlyvee y \in \mathcal{C}_{\Lambda, x} \cap (x+K(\Lambda))$ (resp. $z = x \curlywedge y \in \mathcal{C}_{\Lambda, x} \cap (x-K(\Lambda))$), and by Lemma~\ref{inc} $H(z) \geq H(y) > H(x)$ (resp. $H(z) \leq H(y) < H(x)$). Since $(x+K(\Lambda)) \cap \mathcal{C}_{\Lambda, x}$ (resp. $(x-K(\Lambda)) \cap \mathcal{C}_{\Lambda, x}$) is convex, it includes $[x,z]$. By the intermediate value theorem there exists $w \in [x,z]$ such that $H(w) = H(y)$. By construction $w \in P(x, y)$. 
\end{proof}\\

\begin{lemma}
\label{lembounded}
Consider any CCP cone $K \subseteq \mathbb{R}^n_{\geq 0}$ and some vector $p \in \mathrm{int}( K^*)$. Given any point $x \in \mathbb{R}^n$ and any constant $t > 0$, the set $(x + K) \cap \{y\,:\,p^Ty = p^Tx + t\}$ is bounded. Similarly, $(x - K) \cap \{y\,:\,p^Ty = p^Tx - t\}$ is bounded.
\end{lemma}
\begin{proof}
The first statement will be proved; the proof of the second is similar. With fixed $x$ and $p$, define 
\[
R = \inf_{y\in K,|y| = 1}p^Ty.
\]
Since $p^T(y-x) > 0$ (as $p \in K^*$ and $y-x \in K\backslash\{0\}$), $R > 0$ as the infimum of a positive function on a compact set. Consider any sequence $y_n$ in $(x + K) \cap \{y\,:\,p^Ty = p^Tx + t\}$. We then have $t = p^T(y_n-x) \geq R|y_n-x|$, i.e., $|y_n-x| \leq t/R$, and so $(y_n)$ is bounded.
\end{proof}\\

\begin{lemma}
\label{compconv}
Assume that matrices $\Lambda$, $\Theta$ and $\Gamma$ satisfy Condition A3. $P(x, y)$ is a nonempty compact, convex set for any $y \sim^\Lambda x$.
\end{lemma}
\begin{proof}
It has been shown in Lemma~\ref{foliate} that $P(x, y)$ is nonempty. Either $x=y$, in which case $P(x, y) = \{x\}$ which is trivially compact and convex; or exactly one of $(x+K(\Lambda))\cap \mathcal{C}_{\Gamma,y}$ or $(x - K(\Lambda)) \cap \mathcal{C}_{\Gamma,y}$ is nonempty (since stoichiometry classes are antichains by Corollary~\ref{unord1}). For definiteness assume that $P(x, y) = (x+K(\Lambda))\cap \mathcal{C}_{\Gamma,y}$ is nonempty (the other case is similar). $P(x, y)$ is thus closed and convex as the intersection of closed, convex sets. Applying Lemma~\ref{lembounded} with $K=K(\Lambda)$ and $p=p_\Theta$, $P(x, y)$ is bounded. 
\end{proof}\\

\begin{lemma}
\label{haseq}
Consider System (\ref{eqmain}) satisfying conditions A1--A3. Choose some $e \in E$. Then each stoichiometry class in $\mathcal{C}_{\Lambda,e}$ contains an equilibrium. 
\end{lemma}

\begin{proof}
By Lemma~\ref{compconv}, for arbitrary $x \sim^\Lambda e$, the intersections $P(e, x)$ are nonempty, compact, convex sets. Moreover these sets are forward invariant under $\phi$: $((e+K(\Lambda)) \cup (e-K(\Lambda)))$ is forward invariant by monotonicity of $\phi$ (Corollary~\ref{hscor}) and since the stoichiometry class $\mathcal{C}_{\Gamma,x}$ is forward invariant, $P(e, x)$ is forward invariant as the intersection of forward invariant sets. By the Brouwer fixed point theorem (\cite{spanier} for example) $P(e, x)$ contains an equilibrium. Since $x$ was arbitrary, each stoichiometry class in $\mathcal{C}_{\Lambda,e}$ contains an equilibrium. 
\end{proof}\\

\begin{lemma}
\label{localhomeo}
Consider System (\ref{eqmain}) satisfying conditions A1--A4 and some $e \in E \cap \mathrm{int}(\mathbb{R}^n_{\geq 0})$. There exists an equilibrium $e_0 \in (e - \mathrm{relint}(K(\Lambda))) \cap \mathrm{int}(\mathbb{R}^n_{\geq 0})$. Given any such equilibrium $e_0$ there exists a homeomorphism 
\[
\psi:[H(e_0), H(e)] \to \psi([H(e_0), H(e)]) \subseteq E
\]
such that (i) $H(\psi(h)) = h$, (ii) $\psi([H(e_0), H(e)]) \subseteq \mathrm{int}(\mathbb{R}^n_{\geq 0})$ and (iii) $h_1 < h_2 \Rightarrow \psi(h_1) \llcurly \psi(h_2)$.
\end{lemma}
\begin{proof}
Certainly, $e$ is a not a minimal element of $\mathcal{C}_{\Lambda, e}$ since it lies in $\mathrm{int}(\mathbb{R}^n_{\geq 0})$, and moreover $e$ lies in $\mathrm{relint}(\mathcal{C}_{\Gamma,e})$. Define $R$ as in the proof of Lemma~\ref{lembounded}, fixing $K = K(\Lambda)$ and $p = p_\Theta$, namely: 
\[
R = \inf_{y\in K,|y| = 1}p_\Theta^Ty.
\]
Let $0 < \delta$ be the minimum distance from $e$ to $\partial \mathbb{R}^n_{\geq 0}$ and choose some positive $\epsilon < \delta R$. Then for any $x \preceq e$ such that $H(e)-H(x) \leq \epsilon$, $P(e, x) \subseteq \mathrm{int}(\mathbb{R}^n_{\geq 0})$: this follows since $|y-e| \leq (H(e)-H(x))/R < \delta$ for $y \in P(e, x)$ (see the proof of Lemma~\ref{lembounded}). By Lemma~\ref{compconv}, $P(e, x)$ is a nonempty, compact convex set, and consequently contains an equilibrium $e_{H(x)}$. Since $e_{H(x)} \in (e - \mathrm{relint}(K(\Lambda))) \cap \mathrm{int}(\mathbb{R}^n_{\geq 0})$, it is, by Lemma~\ref{lemord0}, the unique equilibrium in $\mathcal{C}_{\Gamma,x}$. Defining $e_0 = e_{H(e) - \epsilon}$, for $h \in [H(e_0), H(e)]$, the map $\psi: h \mapsto e_h$ is thus well defined, has image in $\mathrm{int}(\mathbb{R}^n_{\geq 0})$ and is clearly a bijection onto its image. $\psi^{-1}$ is continuous as it is simply the restriction of $H$ to the image of $\psi$. 

We next show that $\psi$ is continuous (see also the proof of Lemma~5.12 in \cite{banajiangeli}). Consider any $h \in [H(e_0), H(e)]$, a sequence of values $h_i \subseteq [H(e_0), H(e)]$ with $h_i \to h$, and the corresponding equilibria $e_i = \psi(h_i)$. Since all $e_i$ lie in the order interval $[[e_0, e]] = \{y \in \mathbb{R}^n\,:\, e_0 \preceq y \preceq e\} \subseteq \mathrm{int}(\mathbb{R}^n_{\geq 0})$ which is easily seen to be bounded, $e_i$ contains no unbounded subsequences. Consider any convergent subsequence of $(e_i)$, say $e_{i_k} \to \tilde{e}$. By closure of $E$ and $[[e_0, e]]$, $\tilde{e} \in E \cap [[e_0, e]]$, and by continuity of $H$, $H(\tilde{e}) = h$. Since $\psi(h)$ is the unique equilibrium satisfying these requirements, $\tilde{e} = \psi(h)$. Thus $\psi$ is continuous. 

Finally, that $H(\psi(h)) = h$ is immediate from the definition, and that $h_1 < h_2 \Rightarrow \psi(h_1) \llcurly \psi(h_2)$ is immediate from Lemma~\ref{lemord0} since $\mathrm{Im}(\psi)\subseteq \mathrm{int}(\mathbb{R}^n_{\geq 0})$. 
\end{proof}\\

{\bf Remark.} We could equally prove the existence of $e_0 \in (e + \mathrm{relint}(K(\Lambda))) \cap \mathrm{int}(\mathbb{R}^n_{\geq 0})$ and a homeomorphism $\psi:[H(e), H(e_0)] \to \psi([H(e), H(e_0)]) \subseteq E$.

\subsection{Local asymptotic stability of equilibria}

We are now in a position to prove the local stability of all positive equilibria for System (\ref{eqmain}) satisfying conditions A1--A4. 

{\par{\it Proof of Theorem \ref{mainthm0}}. \ignorespaces}
Consider any positive equilibrium $e$. That $e$ is the unique equilibrium on $\mathcal{C}_{\Gamma,e}$ follows from Lemma~\ref{lemord0}. Choose and fix some equilibrium $e_0 \in (e - \mathrm{relint}(K(\Lambda))) \cap \mathrm{int}(\mathbb{R}^n_{\geq 0})$ as in the proof of Lemma~\ref{localhomeo}. Via Lemma~\ref{localhomeo} define a strictly increasing homeomorphism $\psi:[H(e_0),H(e)]\to \psi([H(e_0),H(e)]) \equiv E_{e_0,e} \subseteq E$ such that $\psi(H(e_0))=e_0$ and $\psi(H(e))=e$. Note that $E_{e_0,e} \subseteq \mathrm{int}(\mathbb{R}^n_{\geq 0})$ and is strictly ordered. Define $U = (e_0 + K(\Lambda)) \cap \mathcal{C}_{\Gamma,e}$. For any point $x \in U\backslash\{e\}$, $e_0 \in x - K(\Lambda)$. On the other hand $e \not \in x - K(\Lambda)$. Since $E_{e_0,e}$ is homeomorphic to a line segment (and hence connected), $E_{e_0,e}$ must intersect $x - \mathrm{relbd}\, K(\Lambda)$. Moreover this intersection is unique, otherwise some pair of distinct equilibria in $E_{e_0,e}$ must fail to be strictly ordered. It is also clear that $(e-\mathrm{relbd}\, K(\Lambda)) \cap E_{e_0,e} = \{e\}$. Thus for all $x \in U$, define $Q(x) \equiv (x-\mathrm{relbd}\, K(\Lambda)) \cap E_{e_0,e}$, and $L(x) \equiv H(Q(x))$. We make three claims (the reader may wish to compare Lemmas~5.16,~5.17~and~5.18 in \cite{banajiangeli}): 
\begin{enumerate}
\item $L(e) > L(x)$ for $x \in U\backslash\{e\}$;
\item $Q$, and hence $L$, are continuous on $U$;
\item $L$ increases strictly along nontrivial orbits.
\end{enumerate}
The first statement is immediate since $H(e) > H(z)$ for $z \in E_{e_0,e}\backslash\{e\}$. Since both $K(\Lambda)$ and $E_{e_0,e}$ are closed sets, with intersection at a unique point, it is not hard to see that for $(x_i) \subseteq U$ with $x_i \to x$, $Q(x_i) \to Q(x)$ and the second claim follows. Finally, given $x \in U\backslash\{e\}$, by strong monotonicity (Corollary~\ref{hscor}), $\phi_t(x) \ggcurly \phi_t(Q(x)) = Q(x)$ for $t > 0$; if $z \preceq Q(x)$ then $z \llcurly \phi_{t}(x)$, i.e., $z \ne Q(\phi_{t}(x))$. So $Q(\phi_{t}(x)) \succ Q(x)$, and thus $L(\phi_{t}(x)) > L(x)$, proving the third claim. 

Thus $L$ serves as a Lyapunov function for $\left. \phi\right|_U$, and, by standard arguments, $e$ is locally asymptotically stable relative to $\mathcal{C}_{\Gamma,e}$. In particular, each $x \in \mathrm{relint}\,U$ is attracted to $e$. 
\endproof

\subsection{Global asymptotic stability of equilibria}

For a global result we need to show that the previous local constructions can be extended.

\begin{lemma}
\label{minimal1}
Consider System (\ref{eqmain}) with conditions A1--A4 defining a local semiflow $\phi$. Consider some $c \in \mathrm{int}(\mathbb{R}^n_{\geq 0})$. Suppose $\mathcal{C}_{\Lambda, c}$ contains an infimum $z \preceq \mathcal{C}_{\Lambda, c}$, and that given any $y \in \mathcal{C}_{\Lambda, c} \cap \mathrm{int}(\mathbb{R}^n_{\geq 0})$, $\phi$ has no $\omega$-limit points on $\mathcal{C}_{\Gamma,y} \cap \partial\,\mathbb{R}^n_{\geq 0}$. Then, for each $h \in [H(z), H(c)]$, the stoichiometry class $\mathcal{C}_{\Lambda, c}^h$ contains exactly one equilibrium. For $h \in (H(z), H(c)]$ this equilibrium is in $\mathrm{int}(\mathbb{R}^n_{\geq 0})$.
\end{lemma}
\begin{proof}
Since stoichiometry classes are antichains (Corollary~\ref{unord1}), and $z \preceq \mathcal{C}_{\Lambda, c}$, $\{z\}$ must be an entire stoichiometry class, and consequently (since stoichiometry classes are invariant) $z$ is an equilibrium. Since $\mathcal{C}_{\Lambda, c}$ contains an equilibrium, by Lemma~\ref{haseq}, each stoichiometry class in $\mathcal{C}_{\Lambda, c}$ contains an equilibrium. Since $\mathcal{C}_{\Gamma,c}$ is a nontrivial stoichiometry class, by the assumption on $\omega$-limit sets of $\phi$ all equilibria in $\mathcal{C}_{\Gamma,c}$ lie in $\mathrm{int}(\mathbb{R}^n_{\geq 0})$. By Lemma~\ref{lemord0}, there is in fact a unique equilibrium in $\mathcal{C}_{\Gamma,c}$. In summary, $\mathcal{C}_{\Gamma,c}$ contains a unique equilibrium and this equilibrium lies in $\mathrm{int}(\mathbb{R}^n_{\geq 0})$.

Choose any $h \in (H(z), H(c)]$. By continuity of $H$, there exists $x \in (z, c]$ such that $H(x) = h$, i.e. $x \in \mathcal{C}_{\Lambda, c}^h$. By basic properties of convex sets, $x \in \mathrm{int}(\mathbb{R}^n_{\geq 0})$, and so $\mathcal{C}_{\Gamma, x}$ is nontrivial. Applying to $\mathcal{C}_{\Gamma, x}$ the argument applied to $\mathcal{C}_{\Gamma,c}$, $\mathcal{C}_{\Gamma, x}$ contains exactly one equilibrium, and this equilibrium is in $\mathrm{int}(\mathbb{R}^n_{\geq 0})$.
\end{proof}\\ 

\begin{lemma}
\label{homeo}
Consider any $c \in \mathrm{int}(\mathbb{R}^n_{\geq 0})$ and the nontrivial $\Lambda$-class $\mathcal{C}_{\Lambda, c}$ with the assumptions of Lemma~\ref{minimal1}. There exists a strictly increasing homeomorphism $\psi:[H(z), H(c)] \to \psi([H(z), H(c)]) \subseteq E$ and such that (i) $H(\psi(h)) = h$, (ii) $h_1 < h_2 \Rightarrow \psi(h_1) \llcurly \psi(h_2)$.
\end{lemma}
\begin{proof}
By Lemma~\ref{minimal1} for each $h \in [H(z), H(c)]$, $\mathcal{C}_{\Lambda, c}^h$ contains a unique equilibrium $e_h$ and for $h \in (H(z), H(c)]$, $e_h \in \mathrm{int}(\mathbb{R}^n_{\geq 0})$. Defining $\psi$ by $\psi(h) = e_h$ it is clear that $\psi$ is a bijection. Continuity of $\psi$ and $\psi^{-1}$ now follow as in Lemma~\ref{localhomeo}. That $h_1 < h_2 \Rightarrow \psi(h_1) \llcurly \psi(h_2)$ follows from Lemma~\ref{lemord0}.
\end{proof}\\

We are now finally ready to prove the main global convergence result.

{\par{\it Proof of Theorem \ref{mainthm}}. \ignorespaces}
Consider any $c \in \mathrm{int}(\mathbb{R}^n_{\geq 0})$, the nontrivial stoichiometry class $\mathcal{C}_{\Gamma,c}$ and the nontrivial $\Lambda$-class $\mathcal{C}_{\Lambda, c}$. By Lemma~\ref{minimal}, assumptions A3 and A5 imply that there exists $z \in \mathcal{C}_{\Lambda, c}$ with $z \preceq \mathcal{C}_{\Lambda, c}$. By condition A6, for $y \in \mathrm{int}(\mathbb{R}^n_{\geq 0})$, $\phi$ has no $\omega$-limit points on $\mathcal{C}_{\Gamma,y} \cap \partial\,\mathbb{R}^n_{\geq 0}$. Thus Lemma~\ref{minimal1} applies: for each $h \in [H(z), H(c)]$, $\mathcal{C}_{\Lambda, c}^h$ contains exactly one equilibrium $e_h$ and for $h \in (H(z), H(c)]$, $e_h \in \mathrm{int}(\mathbb{R}^n_{\geq 0})$. Let $e \equiv e_{H(c)}$ be the equilibrium in $\mathcal{C}_{\Gamma,c}$. 

By Lemma~\ref{homeo}, the map $\psi:[H(z), H(c)] \to \psi([H(z), H(c)]) \equiv E_{z,c} \subseteq E$ defined by $\psi(h) = e_h$ is a strictly increasing homeomorphism. We now follow the proof of Theorem~\ref{mainthm0} with $e_0 = z$. Since $z \preceq \mathcal{C}_{\Lambda, c}$, $P(z, c) = (z + K(\Lambda)) \cap \mathcal{C}_{\Gamma,c} = \mathcal{C}_{\Gamma,c}$. One immediate consequence is that $\mathcal{C}_{\Gamma,c}$ is bounded by Lemma~\ref{compconv}. Thus given any $y \in \mathcal{C}_{\Gamma,c}$, $\omega(y)$ (the $\omega$-limit set of $y$) exists.

For $y \in \mathcal{C}_{\Gamma,c}$ we define $Q(y) = (y-\mathrm{relbd}\,K) \cap E_{z,c}$ and construct the Liapunov function $L(\cdot) = H(Q(\cdot))$ defined on all of $\mathcal{C}_{\Gamma,c}$. That $Q$ is well defined, $L(e) > L(y)$ for $y \in \mathcal{C}_{\Gamma,c}\backslash\{e\}$, $Q$, and hence $L$, are continuous on $\mathcal{C}_{\Gamma,c}$ follow as in the proof of Theorem~\ref{mainthm0}. Note that $L$ is strictly increasing, namely $L(\phi_t(y)) > L(y)$ for $t>0$, if $y \in (\mathcal{C}_{\Gamma,c}\backslash\{e\}) \cap \mathrm{int}(\mathbb{R}^n_{\geq 0})$, but if $y \in\mathcal{C}_{\Gamma,c} \cap \partial \mathbb{R}^n_{\geq 0}$ we can only claim that $L$ is nondecreasing, namely $L(\phi_t(y)) \geq L(y)$ for $t>0$: it may occur that $Q(y) = z$, in which case the entire segment $[z,y]$ may lie in $\partial \mathbb{R}^n$ and we only have $\phi_t(y) \succeq \phi_t(z) = z$. Thus $\omega(y)$ lies in $\{e\} \cup (\mathcal{C}_{\Gamma,c} \cap \partial \mathbb{R}^n_{\geq 0})$. But by condition A6, $\phi$ has no $\omega$-limit points on $\partial \mathbb{R}^n_{\geq 0}$, and so $\omega(y) = \{e\}$. Thus all initial conditions in $\mathcal{C}_{\Gamma,c}$ are attracted to $e$.  
\endproof

\section{Conclusions}

A class of CRNs has been described with strong convergence properties under only weak kinetic assumptions. The networks in this class are defined primarily by the existence of a certain factorisation of their Jacobian matrices, and strong connectedness of their DSR graphs. Roughly speaking, the convergence properties of these CRNs spring from the combination of monotonicity of the associated dynamical systems and the existence of integrals of motion. 

A natural question is how one can identify systems of reactions which fall into the class described in this paper. Such identification would begin with deciding whether a matrix $\Gamma$ admits a factorisation as in condition~A3: we sketch how an algorithm for this purpose would proceed. The starting point is to identify collinear rows of $\Gamma$. A partition of the row-set of $\Gamma$ into $r$ maximal collinear subsets provides a candidate for a first factor $\Lambda$ of the correct form, i.e., as in condition A3(i). Given such a $\Lambda$, the second factor $\Theta$ is uniquely determined (since $\mathrm{ker}\,\Lambda$ is trivial), and it can be checked whether $\mathrm{ker}\,\Theta^T$ is one-dimensional, and if so, whether it intersects the interior of some orthant $\mathcal{O}$ in $\mathbb{R}^r$. If this is the case, but $\mathcal{O} \neq \mathbb{R}^r_{\geq 0}$, then defining $S$ to be the diagonal matrix with diagonal entries in $\{-1, 1\}$ which maps $\mathcal{O}$ to $\mathbb{R}^r_{\geq 0}$, we examine $(\Lambda S)\,(S\Theta)$, and check whether $S\Theta$ has at most one positive entry and at most one negative entry in each column. If this is the case, then $(\Lambda S)\,(S\Theta)$ now provides the desired factorisation.

For the class discussed here, monotonicity is with respect to an order defined by a cone with linearly independent extremal vectors. An interesting question is whether it is possible to extend the theory to more general cones thus progressing with the program of identifying CRNs with simple behaviour: while several of the proofs here were simplified by the fact that $\Lambda$-classes were lattices, the results in \cite{banajiangeli,banajimierczynski} suggest that this may not be crucial to the geometric argument.

Finally, it was observed that Example 1 fell into a class which can also be proved to be monotone in ``reaction coordinates'' \cite{angelileenheersontag}. An interesting theme for future work is to work towards a synthesis of the approaches to monotonicity of CRNs in normal (``species'') coordinates and reaction coordinates. 

\section*{Acknowledgements}
The research of both authors was supported by EPSRC grant EP/J008826/1 ``Stability and order preservation in chemical reaction networks''. 

\appendix

\section{Qualitative classes of matrices}
\label{appqualclass}
A real matrix $M$ determines the {\bf qualitative class} $\mathcal{Q}(M)$ consisting of all matrices with the same sign pattern as $M$. Explicitly, $X \in \mathcal{Q}(M)$ if and only if i) $X$ has the same dimensions as $M$, ii) $(M_{ij} > 0) \Rightarrow (X_{ij} > 0)$, iii) $(M_{ij} < 0) \Rightarrow (X_{ij} < 0)$ and iv) $(M_{ij} = 0) \Rightarrow (X_{ij} = 0)$. \\

The closure of $\mathcal{Q}(M)$ is denoted by $\mathcal{Q}_0(M)$, namely $X \in \mathcal{Q}_0(M)$ if and only if i) $X$ has the same dimensions as $M$, ii) $(M_{ij} > 0) \Rightarrow (X_{ij} \geq 0)$, iii) $(M_{ij} < 0) \Rightarrow (X_{ij} \leq 0)$ and iv) $(M_{ij} = 0) \Rightarrow (X_{ij} = 0)$. \\

$\mathcal{Q}_1(M)$ is defined by deleting iv) from the defining properties of $\mathcal{Q}_0(M)$.

\section{Kinetic assumptions}
\label{appkinetic}
Closely following \cite{banajimierczynski} the following natural assumptions are made about the kinetics, namely about the function $v(x)$ in (\ref{eqmain}). $\mathcal{I}_{j, l}$ is the set of indices of chemicals occurring on the left of reaction $j$ and $\mathcal{I}_{j, r}$ is the set of indices of the chemicals occurring on the right of reaction $j$.
\begin{enumerate}
\item[(K1)]
$\Gamma_{ij}(\partial v_j/\partial x_i) \leq 0$, and if $\Gamma_{ij} = 0$ then $\partial v_j/\partial x_i = 0$. More briefly, in the notation of the previous appendix, $\left[\frac{\partial v_i(x)}{\partial x_j}\right] \in \mathcal{Q}_0(-\Gamma^\mathrm{T})$ at each $x$. This condition has been discussed before in \cite{banajiSIAM,leenheer} for example, and is satisfied by all reasonable kinetics provided no chemical occurs on both sides of
a reaction.

\item[(K2)]
If reaction $j$ is irreversible then
\begin{enumerate}
\item[(i)]
$v_j \geq 0$ with $v_j = 0$ if and only if $x_i = 0$ for some $i \in \mathcal{I}_{j, l}$.
\item[(ii)]
If $x_i > 0$ for all $i \in \mathcal{I}_{j,l}$, then $\partial v_j/\partial x_i > 0$ for each $i \in \mathcal{I}_{j,l}$.
\end{enumerate}
\item[(K3)]
If reaction $j$ is reversible then
\begin{enumerate}
\item[(i)]
If $x_{i} = 0$ for some $i \in \mathcal{I}_{j,l}$ (resp. for some $i \in \mathcal{I}_{j,r}$) then $v_j \leq 0$ (resp. $v_j \geq 0$).
\item[(ii)]
If $x_{i} = 0$ for some $i \in \mathcal{I}_{j,l}$ (resp. for some $i \in \mathcal{I}_{j,r}$), then $v_j < 0$ (resp. $v_j > 0$) if and only if $x_{i'} > 0$ for each $i' \in \mathcal{I}_{j,r}$ (resp. for each $i' \in \mathcal{I}_{j,l}$).
\item[(iii)]
If $x_i > 0$ for all $i \in \mathcal{I}_{j,l}$ (resp. for all $i \in \mathcal{I}_{j,r}$) then, for each $i \in \mathcal{I}_{j,l}$ (resp. $i \in \mathcal{I}_{j,r}$),
$\partial v_j(x)/\partial x_i > 0$ (resp. $\partial v_j(x)/\partial x_i < 0$). 
\end{enumerate}
\end{enumerate}
One consequence of K3(iii) is that for a reversible system of reactions $\left[\frac{\partial v_i(x)}{\partial x_j}\right] \in \mathcal{Q}(-\Gamma^\mathrm{T})$ at each $x \in \mathrm{int}(\mathbb{R}^n_{\geq 0})$.

\section{Persistence in the reversible case}
\label{appsiphon}
Given nonempty $S \subsetneq \{1, \ldots, n\}$, define $P^S$ as the  $n \times n$ projection matrix which maps the nonnegative orthant to $\overline{F}_S$, the closure of the elementary face $F_S \subseteq \mathbb{R}^n_{\geq 0}$ (i.e., $P^S_{ij} = 1$ if $i=j \in S$ and $P^S_{ij} = 0$ otherwise). A {\bf mixed-column matrix} is a matrix such that each nonzero column contains both a positive and a negative entry. We note that as shown in \cite{banajimierczynski} condition A2 implies that:
\begin{enumerate}
\item An elementary face $F_S$ of $\mathbb{R}^n_{\geq 0}$ is either repelling or the vector field on $F_S$ is everywhere tangent to $F_S$. In the latter case the set $S^c$ is termed a {\bf siphon} \cite{angelipetrinet, shiusturmfels}. The face $F_{S^c}$ associated with a siphon $S$ will be termed a {\bf siphon face}.

\item For a system of reversible reactions with no chemical appearing on both sides of any reaction, $F_S$ fails to be repelling if and only if $(I-P^{S})\Gamma$ is a mixed-column matrix. In intuitive terms, if the concentration of chemical $i$ from one side of reversible reaction $j$ is zero on $F_S$, while concentrations of all chemicals on the other side of reaction $j$ are nonzero on $F_S$, then $F_S$ must be repelling since chemical $i$ is being produced by reaction $j$. This occurrence manifests as $(I-P^{S})\Gamma$ failing to be a mixed-column matrix. 
\end{enumerate}

\begin{lemma}
\label{lemstoichclass}
Let $S$ be a nonempty proper subset of $\{1, \ldots, n\}$. Given an $n \times m$ matrix $\Gamma$, suppose there exists $w$ satisfying $\Gamma^T(I-P^S)w = 0$ and $(I-P^S)w > 0$. Then for each $c \in \mathrm{int}(\mathbb{R}^n_{\geq 0})$, $\mathcal{C}_{\Gamma,c}$ does not intersect $\overline{F}_S$.
\end{lemma}
\begin{proof}
$\mathcal{C}_{\Gamma,c}$ intersects $\overline{F}_S$ if and only if there exist $y \in \mathbb{R}^m, z \in \mathbb{R}^n_{\geq 0}$ which solve the equation
\[
c+\Gamma y = P^Sz.
\]
Left-multiplying both sides of the above equation by $w^T(I-P^S)$ gives
\[
w^T(I-P^S)c+w^T(I-P^S)\Gamma y = w^T(I-P^S)c = w^T(I-P^S)P^Sz = 0
\]
Since $c \gg 0$ and $(I-P^S)w > 0$ this is a contradiction, and so $\mathcal{C}_{\Gamma,c} \cap \overline{F}_S = \emptyset$. 
\end{proof}\\

The above lemma tells us that if there exists $w$ such that $(I-P^S)w > 0$ and $\Gamma^T(I-P^S)w = 0$ then no nontrivial stoichiometry class intersects $\overline{F}_S$.

\vspace{0.5cm}

\begin{lemma}
\label{lemsiphon1}
Suppose that $\Lambda$ and $\Theta$ are matrices satisfying the conditions in A3, and  $Q$ is a projection which takes $\mathbb{R}^n_{\geq 0}$ to some nontrivial (closed) face of $\mathbb{R}^n_{\geq 0}$. If $J = Q\Lambda \Theta$ is a mixed-column matrix, then there exists $w$ satisfying $Qw > 0$ and $\Theta^T\Lambda^TQw=0$. 
\end{lemma}
\begin{proof}
Suppose some column of $\Lambda^TQ$ is a nonzero multiple of a unit vector $\hat e_k$ and another is an oppositely signed multiple of $\hat e_k$. In particular, suppose $(\Lambda^TQ)_i = t_1\hat e_k$, $(\Lambda^TQ)_j = -t_2\hat e_k$, where $t_1, t_2 > 0$. Clearly $Q\hat e_i = \hat e_i$ and $Q\hat e_j=\hat e_j$ (otherwise $(\Lambda^TQ)_i = 0$, resp. $(\Lambda^TQ)_j = 0$), and the vector $w = t_2\hat e_i + t_1 \hat e_j$ satisfies $Qw = w > 0$ and $\Lambda^TQw = t_2t_1\hat e_k - t_1t_2\hat e_k = 0$. 

Suppose $K(\Lambda^TQ) \supseteq \mathbb{R}^r_{\geq 0}$ or $K(\Lambda^TQ) \supseteq \mathbb{R}^r_{\leq 0}$. If $K(\Lambda^TQ) \supseteq \mathbb{R}^r_{\geq 0}$, then for some $w>0$, $\Lambda^TQw = y_\Theta$, where $0 \ll y_\Theta \in \mathrm{ker}\,\Theta^T$ and so $\Theta^T\Lambda^TQw=0$. Similarly if $K(\Lambda^TQ) \supseteq \mathbb{R}^r_{\leq 0}$, then for some $w>0$, $\Lambda^TQw = -y_\Theta$, and so $\Theta^T\Lambda^TQw=0$.

Thus if there exists no $w$ satisfying $Qw > 0$, $\Theta^T\Lambda^TQw=0$, then each of the following must be true:
\begin{enumerate}
\item Given a column of $\Lambda^TQ$ of the form $t_1\hat e_j$ where $t_1 \not = 0$, there is no column of $\Lambda^TQ$ of the form $t_2\hat e_j$ with $t_1t_2 <0$. 
\item $K(\Lambda^TQ) \not \supseteq \mathbb{R}^r_{\geq 0}$ and $K(\Lambda^TQ) \not \supseteq \mathbb{R}^r_{\leq 0}$.
\end{enumerate}
These together imply that each column of $Q\Lambda$ has fixed sign (i.e., for each $i$, either $(Q\Lambda)_i = 0$, or $(Q\Lambda)_i < 0$, or $(Q\Lambda)_i > 0$), and there exist two columns with different sign. More precisely, defining $S_0, S_+, S_{-} \subseteq \{1, \ldots, r\}$ by $i \in S_0 \Leftrightarrow$ $(Q\Lambda)_i = 0$, $i \in S_+ \Leftrightarrow$ $(Q\Lambda)_i \in \mathbb{R}^n_{\geq 0} \backslash\{0\}$, and $i \in S_- \Leftrightarrow$ $(Q\Lambda)_i \in \mathbb{R}^n_{\leq 0} \backslash\{0\}$, then $\{S_0, S_+, S_{-}\}$ is a partition of $\{1, \ldots, r\}$ and at least two of $S_0, S_+, S_{-}$ are nonempty. By the definition of $\Theta$ there exists some column of $\Theta$, say column $k$, containing nonzero entries $\Theta_{ik}$ and $\Theta_{jk}$ such that $i,j$ are not both in the same member of the partition $\{S_0, S_+, S_{-}\}$ (i.e., such that $(Q\Lambda)_i$ and $(Q\Lambda)_j$ are of different sign in the sense defined above). If this were not the case, then $\Theta$ could be written in block-diagonal form, in which case it is easy to show that $\mathrm{ker}\,\Theta^T$ contains a nonnegative vector which is not a multiple of $y_\Theta$. But then $(Q\Lambda\Theta)_k = \Theta_{ik}(Q\Lambda)_i + \Theta_{jk}(Q\Lambda)_j$ is a nonzero vector all of whose nonzero entries are of the same sign, and hence $Q\Lambda\Theta$ fails to be mixed-column matrix. 
\end{proof}\\

{\par{\it Proof of Lemma \ref{lemsiphon}}. \ignorespaces}
If $F_S$ fails to be repelling and the reactions are reversible, then $(I-P^S)\Gamma$ is a mixed-column matrix. Applying Lemma~\ref{lemsiphon1} with $Q = I-P^S$, there exists $w$ satisfying $(I-P^S)w > 0$ and $\Gamma^T(I-P^S)w = 0$. But then, by Lemma~\ref{lemstoichclass}, $\mathcal{C}_{\Gamma,c}$ does not intersect $\overline{F}_S$ for any $c \gg 0$. In other words no nontrivial stoichiometry class intersects $F_S$. 
\endproof\\

We illustrate the preceding lemmas with an example. Consider the matrices
\[
\Lambda = \left(\begin{array}{rcc}1&0&0\\0&1&0\\1&0&0\\-1&0&0\\0&0&1\end{array}\right), \,\, \Theta = \left(\begin{array}{rrr}1&0&1\\-1&1&0\\0&-2&-2\end{array}\right), \,\, \Gamma = \Lambda\Theta = \left(\begin{array}{rrr}1&0&1\\-1&1&0\\1&0&1\\-1&0&-1\\0&-2&-2\end{array}\right).
\]
Consider the pairs
\[
\begin{array}{ll}
S_1=\{2,3,5\} & w_1=(1,0,0,1,0)^T\\
S_2=\{1,2,5\} & w_2=(0,0,1,1,0)^T\\
S_3=\{2,5\} & w_3=(1,0,1,2,0)^T\\
S_4=\{3,4\} & w_4=(2,2,0,0,1)^T\\
S_5=\{1,4\} & w_5=(0,2,2,0,1)^T\\
S_6=\{4\} & w_6=(1,2,1,0,1)^T\\
S_7=\{3\} & w_7=(3,2,0,1,1)^T\\
S_8=\{1\} & w_8=(0,2,3,1,1)^T
\end{array}
\]
For each $S_i \subseteq \{1,2,3,4,5\}$ listed above, $(I-P^{S_i})\Gamma$ is a mixed-column matrix, and these are the only proper nonempty subsets of $\{1,2,3,4,5\}$ with this property. Consider for example the set $S_1=\{2,3,5\}$, giving
\[
(I-P^{S_1})\Gamma = \left(\begin{array}{ccccc}1&0&0&0&0\\0&0&0&0&0\\0&0&0&0&0\\0&0&0&1&0\\0&0&0&0&0\end{array}\right)\left(\begin{array}{rrr}1&0&1\\-1&1&0\\1&0&1\\-1&0&-1\\0&-2&-2\end{array}\right) = \left(\begin{array}{rrr}1&0&1\\0&0&0\\0&0&0\\-1&0&-1\\0&0&0\end{array}\right)
\]
which is clearly a mixed-column matrix. Further, in each case $w_i$ satisfies $(I-P^{S_i})w_i > 0$ and $\Gamma^T(I-P^{S_i})w_i = 0$. Thus, as predicted by Lemma~\ref{lemsiphon1}, whenever $(I-P^{S})\Gamma$ is a mixed-column matrix, there exists a corresponding vector $r$, satisfying $(I-P^{S})r > 0$ and $\Gamma^T(I-P^{S})r = 0$. Note that in the first three cases $(I-P^{S_i})w_i \in \mathrm{ker}\,\Lambda^T$, while in the rest $\Lambda^T(I-P^{S_i})w_i = (2,2,1)^T \in \mathrm{ker}\,\Theta^T$. 

\section{Checking condition A6(ii) for Example 3}
\label{appEx3}

For the class of kinetics treated here, as described in \cite{angelipetrinet}, siphons can be identified with sets of chemicals for which each ``input reaction'' is also an ``output reaction''; in other words $S$ is a siphon if and only if for each $i \in S$ every reaction able to produce species $i$ requires some chemical $j \in S$. Similar to the process in \cite{angelipetrinet} using Petri net graphs, siphons can readily be identified from examination of the DSR graph, provided that arcs are signed. 

A siphon is {\bf minimal} if it does not contain (strictly) any other siphon. Clearly every siphon face lies in the closure of at least one siphon face associated with a minimal siphon: so if we can identify each minimal siphon $\Sigma$ and prove that nontrivial stoichiometry classes do not intersect $\overline{F}_{\Sigma^c}$, this ensures that nontrivial stoichiometry classes intersect no siphon faces at all. 

Recall that for Example 3 the stoichiometric matrix takes the form:
\[
\Gamma = \left(\begin{array}{rrrr}-1 & 0 & 0 & 1\\-1 & 1 & 0 & 0\\1 & -1 & 0 & 0\\0 & 1 & -1 & 0\\0 & 0 & -1 & 1\\0 & 0 & 1 & -1\end{array}\right)
\]
corresponding to the substrate ordering $[S_1,E,ES_1,S_2,F,FS_2]$. In this case there are exactly three minimal siphons: $\Sigma_1 = \{2,3\}$ corresponding to substrates $\{E, ES_1\}$; $\Sigma_2 = \{5,6\}$ corresponding to substrates $\{F, FS_2\}$; and $\Sigma_3 = \{1,3,4,6\}$ corresponding to substrates $\{S_1, ES_1,S_2, FS_2\}$. To proceed, note that some nontrivial stoichiometry class intersects the closed face $\overline{F}_{\Sigma_i^c}$ if and only if there exists $c \in \mathrm{int}(\mathbb{R}^6_{\geq 0})$ and $z \in \mathbb{R}^6_{\geq 0}$ such that
\[
c + \Gamma y = P^{\Sigma^c_i}z.
\]
However if such $c \in \mathrm{int}(\mathbb{R}^6_{\geq 0})$ and $z \in \mathbb{R}^6_{\geq 0}$ exist, then defining $v_1 = (0,1,1,0,0,0)^T$, $v_2 = (0,0,0,0,1,1)^T$ and  $v_3 = (1,0,1,1,0,1)^T$, we obtain the following contradictions:
\[
0 < v_i^Tc = v_i^Tc + v_i^T\Gamma y = v_i^TP^{\Sigma^c_i}z = 0 \quad (i = 1,2,3).
\]
These contradictions confirm that the equations $c + \Gamma y = P^{\Sigma_i^c}z$ cannot be satisfied for any $c \gg 0$ and $z \geq 0$, and thus no nontrivial stoichiometry class intersects any of the three closed faces $\overline{F}_{\Sigma_i^c}$. 

\bibliographystyle{siam}

\end{document}